\documentclass[11pt, a4paper]{article}
\usepackage{graphicx}
\usepackage[a4paper, total={6in, 9in}]{geometry}
\usepackage{xcolor}
\usepackage{hyperref}
\usepackage{amsmath, amsthm, amssymb}
\hypersetup{hidelinks}
\usepackage{float}
\usepackage{tikz}
\usepackage{subcaption}
\usepackage{multicol}

\usepackage{mathtools}
\usepackage{booktabs}

\theoremstyle{plain}
\newtheorem{theorem}{Theorem}[section]
\newtheorem{proposition}[theorem]{Proposition}
\newtheorem{lemma}[theorem]{Lemma}
\newtheorem{corollary}[theorem]{Corollary}
\newtheorem{conjecture}[theorem]{Conjecture}

\theoremstyle{definition}
\newtheorem{definition}[theorem]{Definition}
\newtheorem{remark}[theorem]{Remark}
\newtheorem{example}[theorem]{Example}

\newcommand{\CC}{\mathbb C}

\newcommand{\NN}{\mathbb N}
\newcommand{\PP}{\mathbb P}

\newcommand{\RR}{\mathbb R}

\newcommand{\ZZ}{\mathbb Z}

\DeclareMathOperator{\mldeg}{mldeg}
\DeclareMathOperator{\Gr}{Gr}
\DeclareMathOperator{\Sym}{Sym}
\title{Matroid Stratification of \\
ML Degrees of Independence Models}
\author{Oliver Clarke, Serkan Hoşten, Nataliia Kushnerchuk, Janike Oldekop}
\date{}
\let\OLDthebibliography\thebibliography
\renewcommand\thebibliography[1]{
  \OLDthebibliography{#1}
  \setlength{\parskip}{0pt}
  \setlength{\itemsep}{0.1pt plus 0.3ex}
}

\begin{document}

\maketitle

\begin{abstract}
    We study the maximum likelihood (ML) degree of discrete exponential independence models and models defined by the second hypersimplex. For models with two independent variables, we show that the ML degree is an invariant of a matroid associated to the model. We use this description to explore ML degrees via hyperplane arrangements. For independence models with more variables, we investigate the connection between the vanishing of factors of its principal $A$-determinant and its ML degree. Similarly, for models defined by the second hypersimplex, we determine its principal $A$-determinant and give computational evidence towards a conjectured lower bound of its ML degree.
\end{abstract}

\medskip

\section{Introduction}

In this paper we study the maximum likelihood (ML) degree \cite{TheMaximumLikelihoodDegree} of discrete exponential models in two instances: independence models and models defined by the second hypersimplex.
Discrete exponential models are closely related to toric varieties \cite[Section 6.2]{AlgebraicStatistics}.  Let 
$$ A = \begin{pmatrix} a_1 & a_2 & \cdots & a_n
\end{pmatrix},$$ 
where $a_i \in \ZZ^d$, be a $d \times n$ integer matrix whose row-span contains the all-ones vector.  For $w \in (\CC^*)^n$, the \textit{scaled toric variety} $X_{A, w}$ is the Zariski closure of the image of
\[
\psi_{A,w} \colon (\CC^*)^d \to (\CC^*)^n, \quad (\theta_1, \dots, \theta_d) \mapsto (w_1\theta^{a_1}, w_2\theta^{a_2}, \dots w_n \theta^{a_n}).
\]
The \emph{toric variety} is $X_A \coloneqq X_{A,(1,1, \ldots, 1)}$. The corresponding discrete exponential model is the intersection of $X_{A,w}$
with the $(n-1)$-dimensional probability simplex
$$ \left\{ (p_1, \ldots, p_n) \, : \, \sum_{i=1}^n p_i = 1, \mbox{  and  } p_i \geq 0, \, i=1,\ldots, n \right\}.$$
We write $p := (p_1, \dots, p_n)$. The \textit{ML degree} of $X_{A,w}$, denoted $\textup{mldeg}(X_{A,w})$, is the number of complex critical points of the \textit{log-likelihood function}
\[
\ell_u(p) := u_1 \log(p_1) + u_2 \log(p_2) + \dots + u_n \log(p_n)
\]
on $(X_{A,w})_{\textup{reg}} \setminus \mathcal{H}$ for generic data $u \in \NN^n$. Here $(X_{A,w})_{\textup{reg}}$ is the regular locus, and $\mathcal{H}$ is the hypersurface that is the union of the coordinate hyperplanes and the hyperplane defined by $p_1 + \ldots + p_n = 0$.
The ML degree of $X_{A,w}$ is also the number of complex solutions to the \textit{likelihood equations} given by
$$ Ap = Au 
\quad \text{and} \quad
p \in X_{A,w}.$$
For generic $w$, the ML degree of $X_{A,w}$  coincides with the degree of $X_A$, which is the normalized volume of 
the polytope $P = \textup{conv}(a_1, \ldots, a_n)$ \cite[Theorem 3.2]{LikelihoodGeometry}. However, if $w$ lies in the zero locus of the \textit{principal $A$-determinant} of $X_A$, then the ML degree of $X_{A,w}$ is strictly less than the degree of $X_A$.

\subsection{\texorpdfstring{$A$}{A}-discriminant and principal \texorpdfstring{$A$}{A}-determinant}
Let $f_w = \sum_{i=1}^n w_i \theta^{a_i}$ and let
\begin{equation*}
\nabla_A \coloneqq \overline{ \left \{ w \in (\mathbb{C}^*)^n \mid \exists \theta \in (\mathbb{C}^*)^d \textup{ such that } f_w (\theta) = \frac{\partial f_w}{\partial \theta_i}(\theta) = 0 \textup{ for all } i \right \}}.
\end{equation*}
The variety $\nabla_A$ parametrizes those hypersurfaces $\{ f_w  = 0 \}$ that have singular
points in $(\mathbb{C}^*)^d$. 
If $\nabla_A$ has codimension one in $(\mathbb{C}^*)^n$, then the \emph{$A$-discriminant}, denoted $\Delta_A (f_w)$, is defined to be the irreducible polynomial that vanishes on $\nabla_A$. It is unique up to multiplication by a scalar. See \cite[Chapter 9]{Gelfand1994} for more information on the $A$-discriminant.

If the toric variety $X_A$ is smooth, the \emph{principal $A$-determinant} is
\begin{equation} \label{principleAdeterminant}
E_A (w) = \prod_{\Gamma \textup{ face of } P} \Delta_{\Gamma \cap A}
\end{equation}
where the product is taken over all nonempty faces $\Gamma \subseteq P = \textup{conv}(a_1, \ldots, a_n)$ including $P$ itself and $\Gamma \cap A$ is the matrix whose columns are the lattice points of $\Gamma$ \cite[Chapter 10, Theorem 1.2]{Gelfand1994}. If $X_A$ is not smooth, the radical of its principal $A$-determinant is 
given by (\ref{principleAdeterminant}). Our point of departure is the following theorem. 
\begin{theorem}[{\cite[Theorem~2]{amendola2019maximum}}]\label{thm: ML degree drop}
For fixed $w \in (\mathbb{C}^*)^n$, let $X_{A,w}$ be the scaled toric variety. Then $\mldeg(X_{A,w}) < \deg(X_{A,w})$ if and only if $E_A (w) = 0$.
\end{theorem}

\subsection{Results}

An outstanding open problem is to understand the 
variation of $\mldeg(X_{A,w})$ as $w \in (\mathbb{C}^*)^n$ varies. We consider this problem 
for scaled toric varieties $X_{A,w}$ associated to
lattice polytopes that are products of standard simplices as well as matroid base polytopes. 

In the former case, the toric variety $X_A$ is the Segre embedding of a product of projective spaces. We will study the Segre embedding of $\PP^{m-1} \times \PP^{n-1}$ and that of $\PP^1 \times \PP^1 \times \PP^{n-1}$ in detail. The first family corresponds to 
discrete exponential models of two independent discrete random variables with state spaces of size $m$ and $n$, respectively. The second family consists
of independence models of two binary random variables together with a third whose state space has size $n$. In both cases, $\mldeg(X_{A,(1,\ldots,1)})=1$. In the context of Segre embeddings, we denote the scaled toric variety $X_{A,w}$ by $X_w$. 
\begin{definition} \label{def: matroid of w}
    For each scaling $w \in (\CC^*)^{m \times n}$ we define $\widehat w = [I_m \mid w] \in \CC^{m \times (m+n)}$ to be the juxtaposition of the identity matrix with $w$.  We denote by $M_w$  the linear matroid defined by the columns of $\widehat w$.
\end{definition}
The main theorem for the variation of the ML degree 
for scaled Segre embedding of $\PP^{m-1} \times \PP^{n-1}$ is the following. 

\begin{theorem}\label{thm: ML degree is a matroid invariant}
    The ML degree of the Segre embedding $X_w$ of $\PP^{m-1} \times \PP^{n-1}$ is a matroid invariant. More precisely, if $v, w \in (\CC^*)^{m \times n}$ and $M_v \cong M_w$ are isomorphic, then $\mldeg(X_v) = \mldeg(X_w)$. Moreover, the ML degree of $X_w$ is equal to the beta invariant of $M_w$.
\end{theorem}
We will prove this theorem by showing that the ML degree of $X_w$ is equal to the Euler characteristic of the \textit{hyperplane arrangement} associated to $M_w$. A related open problem is to decide whether for each $1 \leq k \leq \deg(X_A)$ there exists a scaling $w$
such that $\mldeg(X_{A,w}) = k$. We resolve this problem with a positive answer for the Segre embedding of $\PP^{m-1} \times \PP^{n-1}$  when
$m =2, 3, 4$. We conjecture that this result holds for arbitrary $m$ and $n$. 
\begin{theorem} \label{thm: all ML degrees are attained} Let $m$ be equal to $2,3$,or $4$. Then
for each $1 \leq k \leq \binom{n+m-2}{m-1}$ there exists a scaling matrix $w \in (\RR^*)^{m \times n}$ such that for the scaled Segre embedding $X_w$ of $\PP^{m-1} \times \PP^{n-1}$, $\mldeg(X_w) = k$.
\end{theorem}
The toric variety $X_A$ that is the Segre embedding of $\PP^1 \times \PP^1 \times \PP^{n-1}$  has degree $n^2+n$. We first determine the principal $A$-determinant. 
\begin{theorem} \label{thm:principal A-det P1xP1xPn}
    Let $w = (w_{ijk})$ be a $2 \times 2 \times n$ tensor of variables where $i,j \in [2]$ and $k \in [n]$. The principal $A$-determinant for $\PP^1 \times \PP^1 \times \PP^{n-1}$ is the product of all $2$-minors arising from all slices of $w$ together with all $2\times 2 \times 2$ and $2\times 2 \times 3$ hyperdeterminants.
\end{theorem}
Our computations indicate that, as in Theorem \ref{thm: ML degree is a matroid invariant}, if two
scaling tensors $v$ and $w$ cause the same set of factors of the principal $A$-determinant to vanish then $\mldeg(X_v) = \mldeg(X_w)$. In particular, we report necessary conditions that guarantee the vanishing of $2\times 2 \times 2$ hyperdeterminant factors based on the vanishing of certain $2$-minors (Proposition \ref{prop:hyperdet vanishes 1}) and the vanishing of $2\times 2 \times 3$ hyperdeterminant factors based on the vanishing
of certain $2 \times 2 \times 2$ hyperdeterminants (Proposition \ref{prop:hyperdet vanishes 2}).

In the final section, we turn to the toric variety defined by the second hypersimplex $\Delta_{2,d}$, namely, the matroid base polytope of the matroid $U_d^2$, the uniform matroid of rank two on $d$ elements. The second hypersimplex $\Delta_{2,d}$ is equal to $\textup{conv}\left( e_i + e_j\, : \, 1\leq i< j \leq d \right)$ where $e_i$ denotes the $i$th standard unit vector in $\RR^d$. In Theorem \ref{thm: A-det for U2n} we compute the principal $A$-determinant of $\Delta_{2,d}$. We report that the scaled
toric variety $X_{2,4}^w$ has only $3$ and $4$ as possible ML degrees (Proposition \ref{prop:3DFace} and Corollary \ref{cor: singularity}). Similarly, we compute 
that the ML degree of $X_{2,5}^w$
attains all values $6 \leq k \leq 11$; see Table \ref{table: MLDegreeStratification of U52}. In conclusion, we offer the following conjecture. 
\begin{conjecture}
The minimum ML degree of the scaled 
toric variety $X_{2,d}^w$ corresponding to 
the second hypersimplex $\Delta_{2,d}$  is  $\binom{d-1}{2}$.
\end{conjecture}

All associated code and detailed information about the computations in this paper are available at the mathematical research data repository \texttt{MathRepo} of the Max-Planck Institute of Mathematics in the Sciences at
\begin{center}
\url{https://mathrepo.mis.mpg.de/MLDegreeStratification}.
\end{center}

\section{Independence Models}

In this section, we consider the ML degree of independence models. These models are given by the scaled Segre embedding of a product of projective spaces. Let $n$ and $m$ be positive integers and let $w\in(\CC^*)^{m \times n}$ be 
a scaling. We define the \textit{independence model} $X_w$ as the Zariski closure of the image of the map 
\[
\psi_w:\PP^{m-1}\times\PP^{n-1}\rightarrow \PP^{mn-1},
\left((s_i)_{i \in [m]}, (t_j)_{j \in [n]}\right) \mapsto 
(w_{ij} s_i t_j)_{(i,j) \in [m] \times [n]}.
\]
The variety $X_w$ is a scaled Segre embedding of $\PP^{m-1}\times\PP^{n-1}$ with respect to $w$. The coordinates of $X_w$ are indexed by pairs $(i,j) \in [m] \times [n]$ and are denoted $p_{ij} = w_{ij}s_it_j$. This 
projective toric variety arises from a 
$0/1$ matrix $A$ with $(n + m)$ rows and $nm$ columns. The columns of $A$ are the vertices of the polytope $P = \Delta_{m-1} \times \Delta_{n-1}$, which is the product of two standard simplices. The ideal of $X_w$ in $\CC[p_{ij} : i \in [m], \, j \in [n]]$ is generated by the binomials
\[
\frac{p_{ij}p_{k\ell}}{w_{ij}w_{k\ell}}-\frac{p_{i\ell}p_{kj}}{w_{i\ell}w_{kj}} 
\text{ for all } i, k\in [m] \text{ and } j, \ell\in [n].
\]
The degree of $X_w$ is $\deg(X_w) = \binom{n+m-2}{m-1}$. If $w$ is generic, then $\mldeg(X_w) = \deg(X_w)$. To study how the ML degree drops for different scalings, we recall the principal $A$-determinant. 
\begin{proposition}[{\cite[Chapter~10.1]{Gelfand1994}}]\label{prop: Segre A det}
    The principal $A$-determinant of $\PP^{m-1} \times \PP^{n-1}$ is the product of all minors of $w$:
    \[
    E_A = 
    \prod_{i=1}^{\min(m, n)}
    \prod_{\substack{1\leq a_1<\ldots<a_i\leq m;\\ 
    1\leq b_1<\ldots<b_i\leq n}}[a_1, \ldots, a_i; b_1,\ldots, b_i],
    \]
    where $[a_1, \ldots, a_i; b_1,\ldots, b_i]$ denotes a minor defined by rows $a_1, \ldots, a_i$ and columns $b_1, \ldots, b_i$. 
\end{proposition}

\begin{example}[{\cite[Example~27]{amendola2019maximum}}]\label{example: all ML degrees for P2xP2}   
    Consider the Segre embedding of $\PP^2 \times \PP^2$ given by the matrix
    \[
        A = 
        \begin{pmatrix}
            1&1&1&0&0&0&0&0&0\\
            0&0&0&1&1&1&0&0&0\\
            0&0&0&0&0&0&1&1&1\\
            1&0&0&1&0&0&1&0&0\\
            0&1&0&0&1&0&0&1&0\\
            0&0&1&0&0&1&0&0&1
        \end{pmatrix}.
    \]
The degree of $X_w$ is $\deg(X_w) = 6$. For each $d \in [6]$, there exists a scaling $w$ such that $\mldeg(X_w) = d$. Such scalings $w$ are computed in \cite{amendola2019maximum} and are shown in Table~\ref{tab: all ML degrees for P2xP2}.
\begin{table}[t]
    \centering
    \begin{tabular}{ccccccc}
    \toprule
        ML degree & $1$ & $2$ & $3$ & $4$ & $5$ & $6$ \\
    \midrule
        scaling $w$ &
        $
        \begin{bmatrix}
            1 & 1 & 1 \\
            1 & 1 & 1 \\
            1 & 1 & 1
        \end{bmatrix}
        $
        & 
        $
        \begin{bmatrix}
            1 & 2 & 1 \\
            1 & 1 & 1 \\
            1 & 1 & 1
        \end{bmatrix}
        $
        &
        $
        \begin{bmatrix}
            1 & 2 & 3 \\
            1 & 1 & 1 \\
            1 & 1 & 1
        \end{bmatrix}
        $
        &
        $
        \begin{bmatrix}
            1 & 2 & 3 \\
            1 & 2 & 1 \\
            1 & 1 & 1
        \end{bmatrix}
        $ 
        & 
        $
        \begin{bmatrix}
            1 & 2 & 3 \\
            2 & 1 & 1 \\
            1 & 1 & 1
        \end{bmatrix}
        $
        &
        $
        \begin{bmatrix}
            1 & 2 & 3 \\
            2 & 3 & 1 \\
            1 & 1 & 1
        \end{bmatrix}
        $\\
    \bottomrule
    \end{tabular}
    \caption{Scalings for $\PP^2 \times \PP^2$ that achieve all possible ML degrees.}
    \label{tab: all ML degrees for P2xP2}
\end{table}
\end{example}


\subsection{Matroid stratification and ML degrees}

In this section, we characterize the ML degree of $X_w$ as a matroid invariant of a matroid associated to $w$. Moreover, it was conjectured by Am\'endola et al. \cite[Conjecture~28]{amendola2019maximum} that the ML degree of $X_w$ can attain all values between $1$ and $\deg(X_w)$.  We use our characterization of the ML degree in Section~\ref{section: ML degree via hyperplane arrangements} to 
prove the conjecture in several cases of small dimension. First, let us recall the matroid stratification of the Grassmannian.

\medskip

The Grassmannian $\Gr(m,m+n) \subset \PP^{\binom{m+n}{m}-1}$ is the projective variety of $m$-dimensional linear subspaces of $\CC^{m+n}$. 
Each such linear subspace is given by the row-span of a $m \times (m+n)$ matrix $V \in \CC^{m \times (m + n)}$. The columns of $V$ define a \textit{representable} (or linear) \textit{matroid} $M$ over $\CC$. The matroid $M$ is determined by its \textit{bases}, which are the maximal linearly independent subsets of the columns of $V$. 
Therefore, there is a well-defined subset $\mathcal S(M) \subseteq \Gr(m,m+n)$ of points of the Grassmannian associated to the matroid $M$. We call the set $\mathcal S(M)$ a \textit{matroid stratum} of the Grassmannian, and the collection of matroid strata constitutes the \textit{matroid stratification} of the Grassmannian.

As in Definition \ref{def: matroid of w}, for
a scaling $w \in (\CC^*)^{m \times n}$, we let 
$\widehat w = [I_m \mid w] \in \CC^{m \times (m+n)}$. By Proposition~\ref{prop: Segre A det}, the principal $A$-determinant $E_A$ of $\PP^{m-1} \times \PP^{n-1}$ is the product of all minors of $w$. Notice that the set of minors of $w$ is exactly the set of maximal minors of $\widehat w$. So, the data of the matroid $M_w$ is equivalent to the data of which factors of $E_A$ vanish when evaluated at $w$. 



We write $x = (x_1 , \dots , x_m)$ for the coordinates of $\PP^{m-1}$ and $y = (y_1 , \dots , y_n)$ for the coordinates of $\PP^{n-1}$. The hyperplane arrangement $H_{w}$ obtained from $w$ lives in $\PP^{m-1}$. Its hyperplanes are the vanishing loci of the linear forms whose coefficients are the columns of $\widehat w$. For each column $w_i$ of $w$, we define the linear form
\[
    \ell_i(x) \, = \, w_{1i} x_1 + w_{2i} x_2 + \dots + w_{mi} x_m.
\]
We let 
\[
U = \{ 
x \in \PP^{m-1} \colon \ell_i(x) \neq 0 \text{ and }
x_i \neq 0 \text{ for all } i \in [m] 
\}
\]
be the complement of the hyperplane arrangement $H_{w}$.
Consider also the arrangement of hyperplanes
\[
\mathcal H = \left\{(p_{ij})_{i,j} \in \PP^{nm-1} \colon \left(\prod_{i,j} p_{ij}\right)\left(\sum_{i,j} p_{ij}\right) = 0\right\}.
\]
\begin{theorem}[{\cite[Theorem~1]{huh_2013}}, {\cite[Theorem~1.7]{LikelihoodGeometry}}]\label{thm: ML degree as Euler characteristic}
If a $d$-dimensional very affine variety $X \backslash \mathcal H$ is smooth, then $\mldeg(X) = (-1)^d \chi(X \backslash \mathcal H)$ is the signed Euler characteristic.
\end{theorem}

As a reminder, a variety is said to be \textit{very affine} if it is a closed subvariety of some algebraic torus. The standard torus of projective space is the image of the embedding
\[
(\CC^*)^{d} \rightarrow \PP^{d}, \quad (\theta_1, \dots, \theta_d) \mapsto (1 , \theta_1 , \dots , \theta_d),
\]
which is equal to $\PP^d \backslash V(p_0p_1\dots p_d)$, the complement of the vanishing locus of the product of coordinates.

In our setting $X_w \backslash \mathcal H$ is clearly smooth, as the product $\PP^{m-1} \times \PP^{n-1}$ is smooth. Also, the variety $X_w \backslash \mathcal H$ is very affine because $X \backslash \mathcal H \subseteq (\CC^*)^{nm-1}$ is contained in the standard torus of $\PP^{nm-1}$. We define the projection $\pi_w : X_w \backslash \mathcal H \rightarrow \PP^{m-1}, \, (w_{ij} x_i y_j)_{i,j} \mapsto (x_i)_i$. Note that its image is contained in the torus $(\CC^*)^{m-1} \subseteq \PP^{m-1}$. We show that each fiber of $\pi_w$ corresponds to a hyperplane arrangement complement. For each subset $S \subseteq [m]$, define
    \[
    X_S = \{ 
    (w_{ij}x_iy_j)_{i,j} \in X_w \backslash \mathcal H \colon \ell_i(x) \neq 0 \text{ for all } i \in [n]\setminus S
    \}.
    \]

\begin{lemma}\label{lem: Euler characteristic by fibration}
    We have $\chi(X_w \backslash \mathcal H) = \chi(X_\emptyset) = \chi(U)$.
\end{lemma}

\begin{proof}
    Throughout this proof, we write $T^d$ for the $d$-dimensional algebraic torus $(\CC^*)^d$. Moreover, given a subset $S \subseteq [n]$, we write $T^S$ for the projective torus $T^{|S| - 1} \subseteq \PP^{|S| - 1}$ whose coordinates are $y_i$ with $i \in S$. By convention, we take $T^\emptyset$ to be a point.
    
    Observe that $X_w \backslash \mathcal H = \bigsqcup_S X_S$ is a disjoint union of subvarieties. We will show that $\chi(X_\emptyset) = \chi(U)$ and $\chi(X_S) = 0$ for all $S \neq \emptyset$. We note that
    \[
    X_w \backslash \mathcal H = \left\{
    (x,y) \in T^{m-1} \times T^{n-1} \colon
    \sum_{i = 1}^n y_i \ell_i(x) \neq 0  
    \right\}.
    \]
    Fix any subset $S \subseteq [n]$ and consider the restriction of $\pi_w$ to $X_S$. For each $x \in T^{m-1}$, the fiber $\pi_w^{-1}(x)$ is the set of points $(x,y)$ that do not kill the linear form $f_x(y) := \sum_i y_i \ell_i(x)$. By assumption, we have $\ell_i(x) = 0$ if and only if $i \in S$. So, for each $i \in S$, the coordinate $y_i$ does not affect the value of $f_x(y)$. Hence
    \[
    \pi_w^{-1}(x) = 
    \{(x,y) \in T^{m-1} \times T^{n-1} \colon f_x(y) \neq 0\} \cong 
    T^{S} \times \left\{(y_i) \in T^{[n] \backslash S} \colon \sum_{i \in [n]\backslash S} y_i \ell_i(x) \neq 0\right\}.
    \]
    Consider the set $Y_S := \left\{(y_i) \in T^{[n] \backslash S} \colon \sum_{i \in [n]\backslash S} y_i \ell_i(x) \neq 0\right\}$ which appears as a factor of the above fiber. If $S = [n]$, then the linear form $f_x(y)$ is identically zero and the fiber $\pi_w^{-1}(x)$ is empty. From now on, we assume $S$ is a proper subset of $[n]$. 
    
    Since $\ell_i(x) \neq 0$ for each $i \in [n] \backslash S$, we have that $Y_S$ is the hyperplane arrangement complement of the coordinate hyperplanes together with a single generic hyperplane. This is the hyperplane arrangement of the uniform matroid on $[n] \backslash S$ of rank $n - |S| - 1$. It follows that the Euler characteristic $\chi(Y_S) = 1$.

    Recall that the Euler characteristic is multiplicative over direct products. Since the Euler characteristic of a proper algebraic torus is zero and $T^{\emptyset}$ is a point, we have
    \[
    \chi\left(\pi_w^{-1}(x)\right) = \chi\left(T^S\right) \cdot \chi\left(Y_S\right) = \chi\left(T^S\right) = \begin{cases}
        1 & \text{if } S = \emptyset,\\
        0 & \text{otherwise}.
    \end{cases}
    \]
    The Euler characteristic is also multiplicative over a fibration. In particular, for the fibration $\pi_w : X_S \rightarrow T^{m-1}$, it follows that $\chi(X_S) = 0$ for non-empty proper subsets $S \subseteq [n]$. Since the Euler characteristic is additive over the disjoint union $X_w \backslash \mathcal H = \bigsqcup_S X_S$, we have $\chi(X_w \backslash \mathcal H) = \chi(X_\emptyset)$. Note that $\pi_w : X_\emptyset \rightarrow T^{m-1}$ is a fibration whose base is the hyperplane arrangement complement $U$. By the above, the fibers of $\pi_w$ have Euler characteristic one. So, by the multiplicative property of the Euler characteristic, we have $\chi(X_w \backslash \mathcal H) = \chi(X_\emptyset) = \chi(U)$.
\end{proof}

\begin{proof}[Proof of Theorem~\ref{thm: ML degree is a matroid invariant}]
    Let $U$ be the hyperplane arrangement complement associated to the matroid $M_w$.
    By Theorem~\ref{thm: ML degree as Euler characteristic} and Lemma~\ref{lem: Euler characteristic by fibration}, we have that $\mldeg(X_w) = (-1)^d \chi(U)$.
    It is well known that the Euler characteristic $\chi(U)$ is a matroid invariant. In particular, it is the beta invariant of $M_w$. So it follows that the ML degree is also a matroid invariant of $M_w$.
\end{proof}

    If $w \in (\RR^*)^{m \times n}$ is real-valued, then the beta invariant has a convenient interpretation.

\begin{proposition}[{\cite[Theorem~D]{greene1983WhitneyNumbers}}, {\cite{zaslavsky1975facing}}]
    Let $w \in \RR^{m \times n}$ and $H \subseteq \PP^{nm - 1}$ be its complex projective hyperplane arrangement. The beta invariant of $H$ is equal to the number of bounded regions of $H \cap R$, where $R$ is a real affine patch of $\PP^{nm - 1}$.
\end{proposition}

Theorem~\ref{thm: ML degree is a matroid invariant} immediately implies the following.

\begin{corollary}\label{cor: ML degree by counting bounded regions}
    Suppose that $w \in (\RR^*)^{n \times m}$ is a real scaling for $\PP^{m-1} \times \PP^{n-1}$. Then the ML degree of $X_w$ is equal to the number of bounded regions of a real affine patch of the hyperplane arrangement $H_w$ of the matroid $M_w$.
\end{corollary}

We recall that matroids are partially ordered via the
\textit{weak order}. Suppose that $M$ and $N$ are matroids on the same ground set. We write $M \le N$ if every independent set of $M$ is an independent set of $N$. In other words, every basis of $M$ is a basis of $N$. We say $M$ is a \textit{specialization}
of $N$.

\begin{example} \label{ex: all matroids for P2 times P2}
In the case of $\PP^2 \times \PP^2$ the matroids that stratify the ML degree are matroids $M_w$ realized by matrices of the form
$$ \begin{pmatrix}
    1 & 0 & 0 & w_{11} & w_{12} & w_{13} \\
    0 & 1 & 0 & w_{21} & w_{22} & w_{23} \\
    0 & 0 & 1 & w_{31} & w_{32} & w_{33} 
\end{pmatrix}$$
where $w_{ij}\neq 0$. The ground set of these matroids have six elements and their rank is three. Up to isomorphism, there are $38$ matroids of rank $3$ on $[6]$, of which there are eight of the form $M_w$.
In Figure \ref{fig: stratification} we illustrate the Hasse diagram of the poset of matroids $M_w$ under the weak order and label the corresponding node by its ML degree.
We note that while all possible ML degrees between $1$ and $6$ are achieved, no single chain realizes every ML degree.

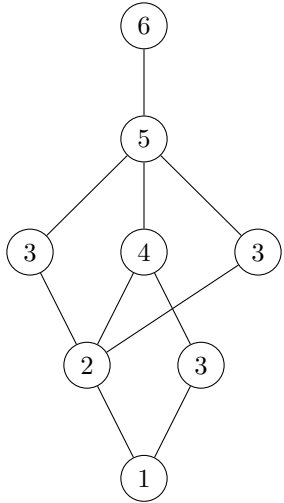
\begin{figure}[t]
\centering
\begin{tikzpicture}
    \node[shape=circle,draw=black,scale=0.9] (A) at (0,0) {6};
    \node[shape=circle,draw=black,scale=0.9] (B) at (0,-1.5) {5};
    \node[shape=circle,draw=black,scale=0.9] (C) at (0,-3) {4};
    \node[shape=circle,draw=black,scale=0.9] (D) at (-1.5,-3) {3};
    \node[shape=circle,draw=black,scale=0.9] (E) at (1.5,-3) {3};
    \node[shape=circle,draw=black,scale=0.9] (F) at (-0.75,-4.5) {2} ;
    \node[shape=circle,draw=black,scale=0.9] (G) at (0.75,-4.5) {3} ;
    \node[shape=circle,draw=black,scale=0.9] (H) at (0,-6) {1} ;
    
    \draw (A) -- (B);
    \draw (B) -- (C);
    \draw (B) -- (D);
    \draw (B) -- (E);
    \draw (C) -- (F);
    \draw (C) -- (G);
    \draw (D) -- (F);
    \draw (E) -- (F);
    \draw (F) -- (H);
    \draw (G) -- (H);
\end{tikzpicture}
\caption{ML degree stratification for $\PP^2 \times \PP^2$}
 \label{fig: stratification}
\end{figure}

We identify these matroids explicitly as follows. Recall the matroids in Example~\ref{example: all ML degrees for P2xP2} and write $M_i$ for the matroid whose corresponding ML degree is $i \in [6]$. These matroids form two maximal chains in Figure~\ref{fig: stratification} given by
\[
(M_1, M_2, M_3, M_5, M_6) 
\quad \text{and} \quad
(M_1, M_2, M_4, M_5, M_6).
\]
So, there are two matroids in the figure that are not given by $M_i$ for some $i$. One is the dual $M_3^*$ and occurs in the maximal chain $(M_1, M_2, M_3^*, M_5, M_6)$. The other matroid is $M_w$ realized by the matrix 
\[
\begin{pmatrix}
    1 & 0 & 0 & 1 & 1 & 1 \\
    0 & 1 & 0 & 1 & 2 & 2 \\
    0 & 0 & 1 & 1 & 2 & 1
\end{pmatrix},
\]
which occurs in the maximal chain $(M_1, M_w, M_4, M_5, M_6)$.
We observe that, with the exception of $M_3$ and $M_3^*$, all matroids $M$ appearing in Figure~\ref{fig: stratification} are \textit{self-dual}, that is, $M^* \cong M$.
\end{example}

In the above example, we have seen that matroids and their duals give the same ML degree. This is no coincidence. It is well-known that the beta invariant of a matroid is equal to the beta invariant of its dual. So, we immediately deduce the following corollary.

\begin{corollary}
    Let $w$ be a scaling for $\PP^{m-1} \times \PP^{n-1}$ and $w'$ be a scaling for $\PP^{n-1} \times \PP^{m-1}$. If $M_w$ is isomorphic to the dual $M_{w'}^*$ then $\mldeg(X_w) = \mldeg(X_{w'})$.
\end{corollary}

The above example also shows that ML degrees respect the weak order on matroids. We finish this subsection with the following conjecture.

\begin{conjecture}
    Suppose that $v$ and $w$ are scalings for $\PP^{m-1} \times \PP^{n-1}$ with $\widehat v \in \mathcal S(M_v)$ and $\widehat w \in \mathcal S(M_w)$. If $M_v < M_w$ then $\mldeg(X_{A,v}) < \mldeg(X_{A,w})$.
\end{conjecture}

\subsection{ML degree via hyperplane arrangements}\label{section: ML degree via hyperplane arrangements}
In this section we prove Theorem \ref{thm: all ML degrees are attained}, which settles \cite[Conjecture~28]{amendola2019maximum} for Segre embeddings of $\PP^{m-1} \times \PP^{n-1}$ when $m=2, 3, 4$. The proof follows from Example \ref{ex: all ml degrees P1xPn}, Corollary \ref{cor: P2xPn-1} and Corollary \ref{cor: P3xPn-1} below.  Moreover, we present a partial result for larger values of $m$. We consider the case of $m=2$ in the next example. 

\begin{example}[ML degree of $\PP^1 \times \PP^{n-1}$]\label{ex: all ml degrees P1xPn}
    Let $X_w$ be the scaled Segre embedding of $\PP^{1} \times \PP^{n-1}$ where $w$ is a $2 \times n$ matrix with entries in $\RR^*$. The extended matrix $\widehat w$ is given by
    \[
    \widehat w = \begin{pmatrix}
        1 & 0 & w_{1,1} & \dots & w_{1,n} \\
        0 & 1 & w_{2,1} & \dots & w_{2,n} 
    \end{pmatrix},
    \]
    which defines the matroid $M_w$. The hyperplane arrangement $H_w \subseteq \PP^1$ is an arrangement of points in $\PP^1$. For each column $(\lambda, \mu)^T$ of $\widehat w$, the corresponding hyperplane is $\lambda x_1 + \mu x_2 = 0$, which gives the point $(\mu,\, -\lambda) \in \PP^1$. Let us consider the real affine patch 
    \[
    R = \{(1, \tau) \in \PP^1 : \tau \in \RR\} \cong \RR.
    \]
    The column $(1, 0)^T$ of $\widehat w$ corresponds to $0 \in \RR$ and the column $(0, 1)^T$ does not appear in $H_w \cap R$. All other columns correspond to non-zero elements of $\RR$ since the entries of $w$ are non-zero. Suppose that the columns of $w$ define $k$ distinct points in $\RR$. Then, together with $0 \in \RR$, the hyperplane arrangement complement $\RR \backslash (H_w \cap R)$ has exactly $k$ bounded regions. 
    Hence, the ML degree of $X_w$ is $k$. For each $k \in [n]$ the following gives a real scaling 
    so that $\mldeg(X_w) = k$:
    \[
    w = 
    \begin{pmatrix}
    1 & 1 & 1 & \dots & 1 & 1 & \dots & 1 \\
    1 & 2 & 3 & \dots & k & 1 & \dots & 1 
    \end{pmatrix}.
    \]
\end{example}

Before we handle the cases $m=3$ and $m=4$, we revisit the example $m=n=3$.
\begin{example}[ML degree of $\PP^2 \times \PP^{2}$]\label{ex: line arrangements}
    Consider Segre embeddings of $\PP^2 \times \PP^2$ with the scalings from Example~\ref{example: all ML degrees for P2xP2} achieving all ML degrees. Let us demonstrate the construction of the corresponding hyperplane arrangements in $\mathbb{R}^2$. For instance, for the following scaling the ML degree of $X_w$ drops by exactly two, i.e., $\mldeg(X_w)=4$:    
    \[
    \widehat w = \begin{pmatrix}
        1 & 0 & 0 & 1 & 2 & 3 \\
        0 & 1 & 0 & 1 & 2 & 1 \\
        0 & 0 & 1 & 1 & 1 & 1
    \end{pmatrix}.
    \]
    To view this as a hyperplane arrangement, we take the columns as the coefficients of linear forms:
    \[
    \ell_1 = x, \, \ell_2 = y, \, \ell_3 = z, \,
    \ell_4 = x+y+z, \, \ell_5 = 2x+2y+z, \, \ell_6 = 3x+y+z.
    \]
    This hyperplane arrangement naturally lives in $\PP^2$. Let us move to the affine patch $(1, \widehat x, \widehat y)$ with new coordinates $\widehat x = y/x$ and $\widehat y = z/x$. The hyperplane $\ell_1 = 0$ is at infinity. The remaining five hyperplanes (lines) form an arrangement with four bounded regions. Figure ~\ref{fig: line_arrangements} shows the arrangements for the representative scalings corresponding to each possible ML degree. 
    
    \begin{figure}[p]

        \begin{subfigure}{0.5\textwidth}
            \includegraphics[scale=0.25]{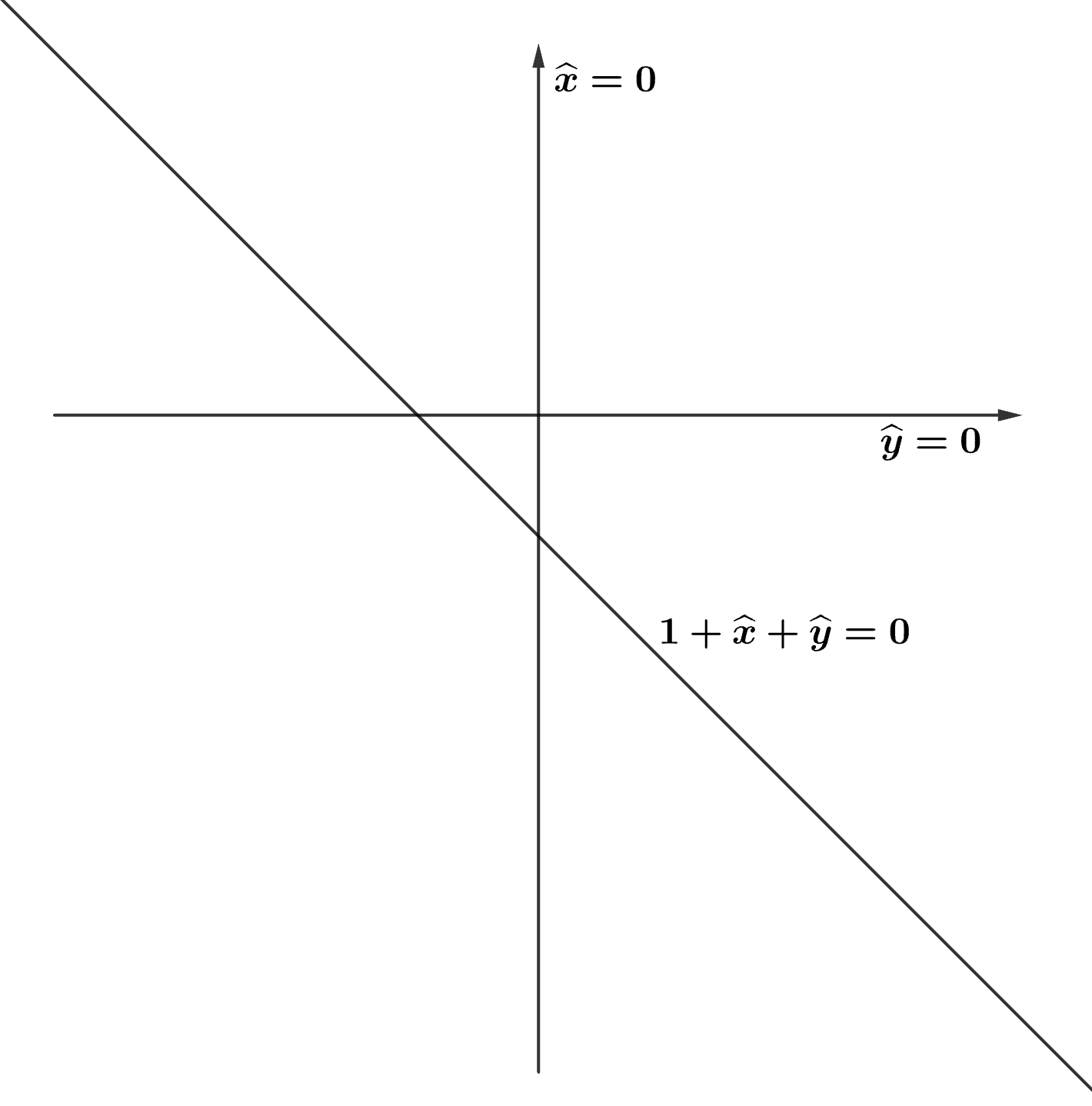}    
        \caption{ML degree 1}
        \end{subfigure}
        \begin{subfigure}{0.5\textwidth}
            \includegraphics[scale=0.25]{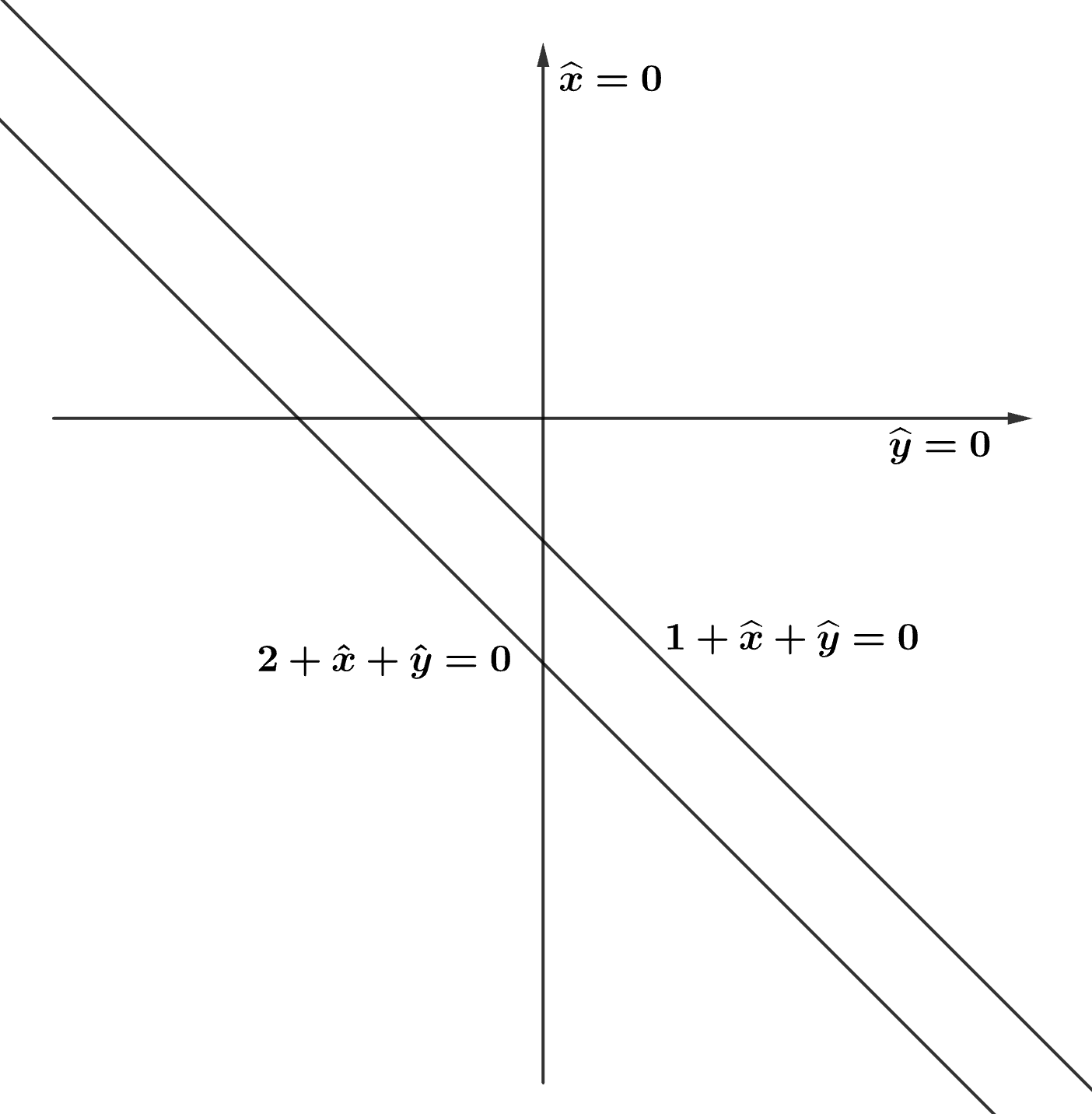}    
        \caption{ML degree 2}
        \end{subfigure}\\
         \begin{subfigure}{0.5\textwidth}
            \includegraphics[scale=0.25]{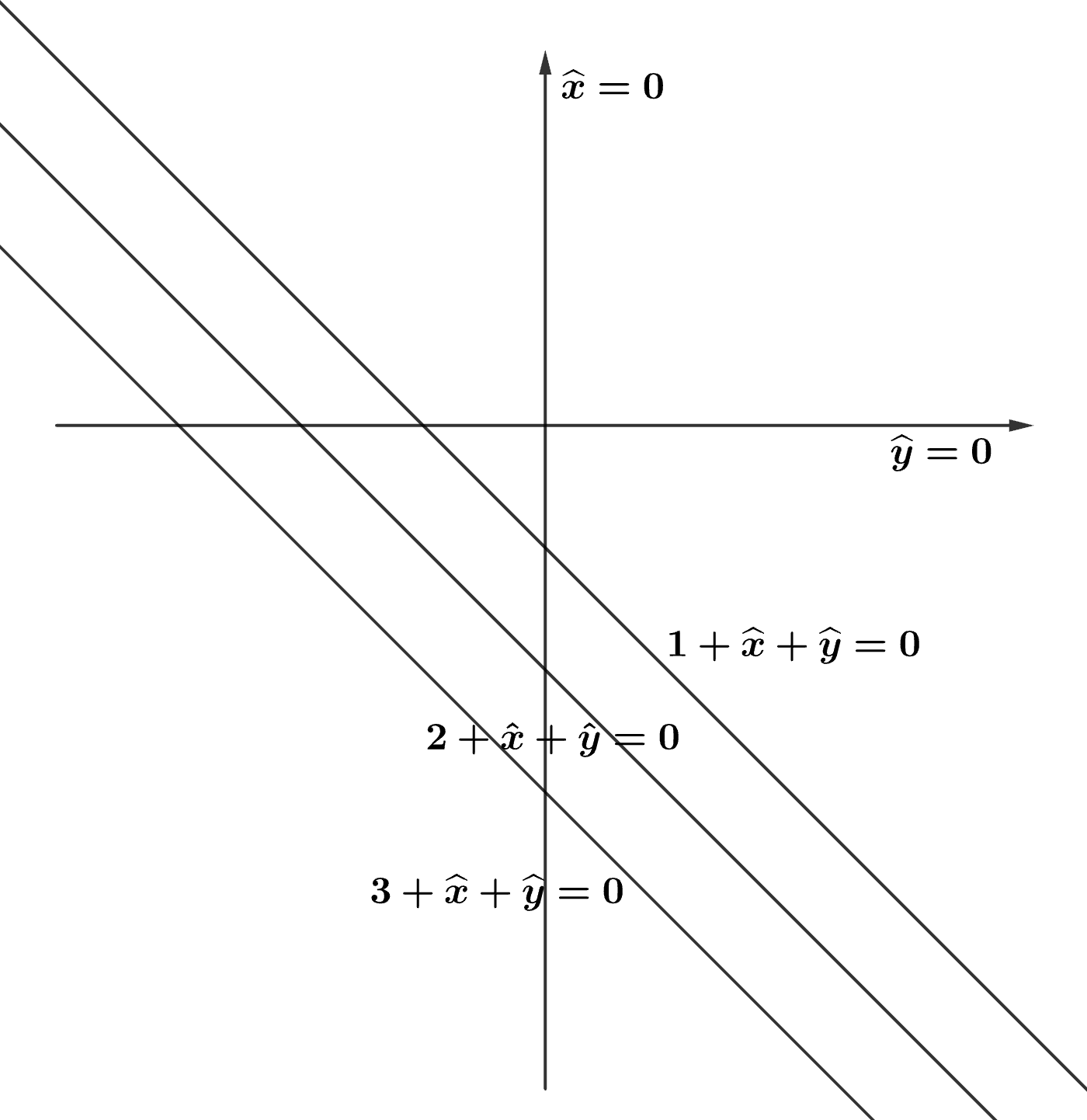}    
        \caption{ML degree 3}
        \end{subfigure}
        \begin{subfigure}{0.5\textwidth}
            \includegraphics[scale=0.25]{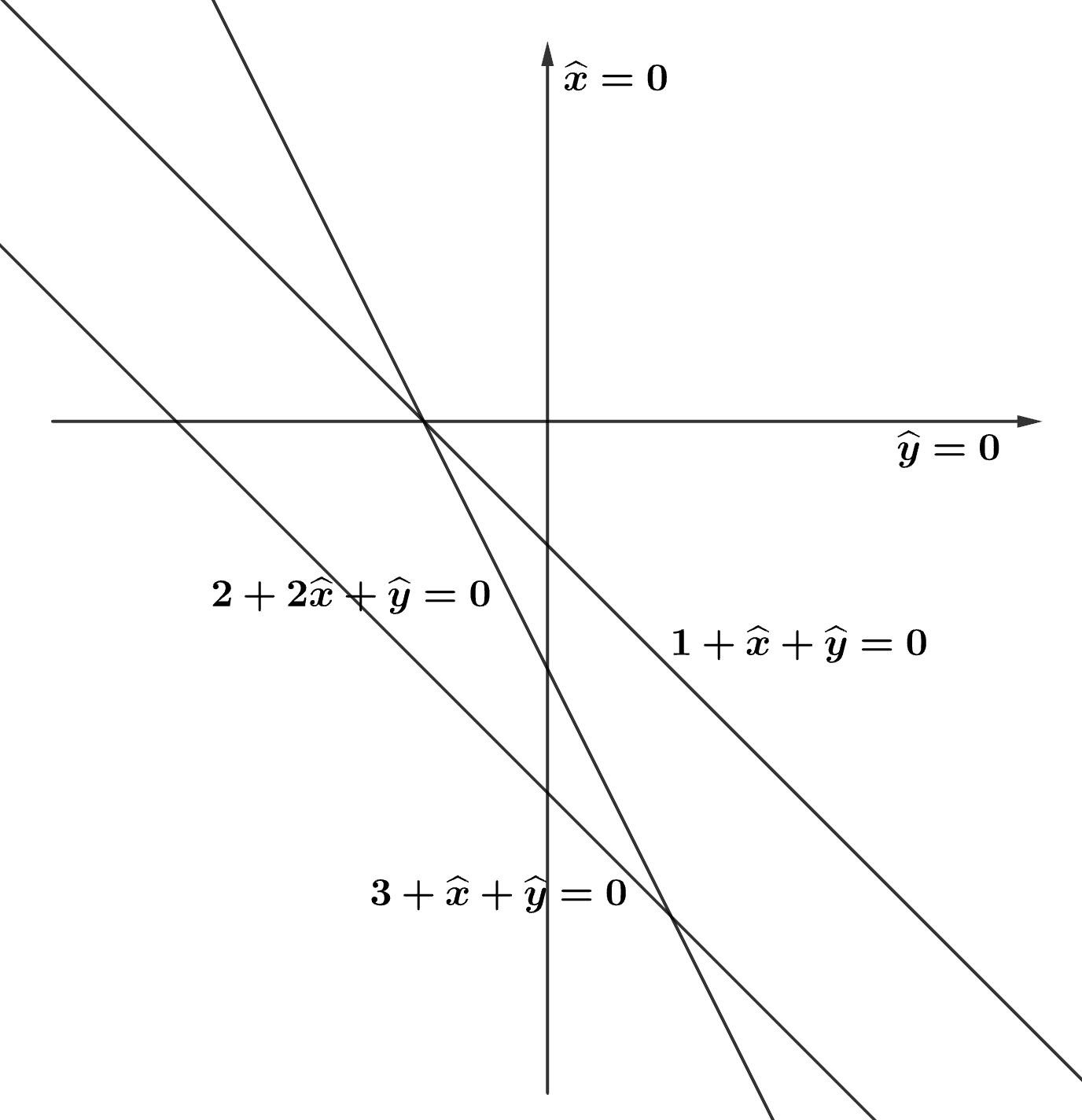}    
        \caption{ML degree 4}
        \end{subfigure}
        \\
         \begin{subfigure}{0.5\textwidth}
            \includegraphics[scale=0.25]{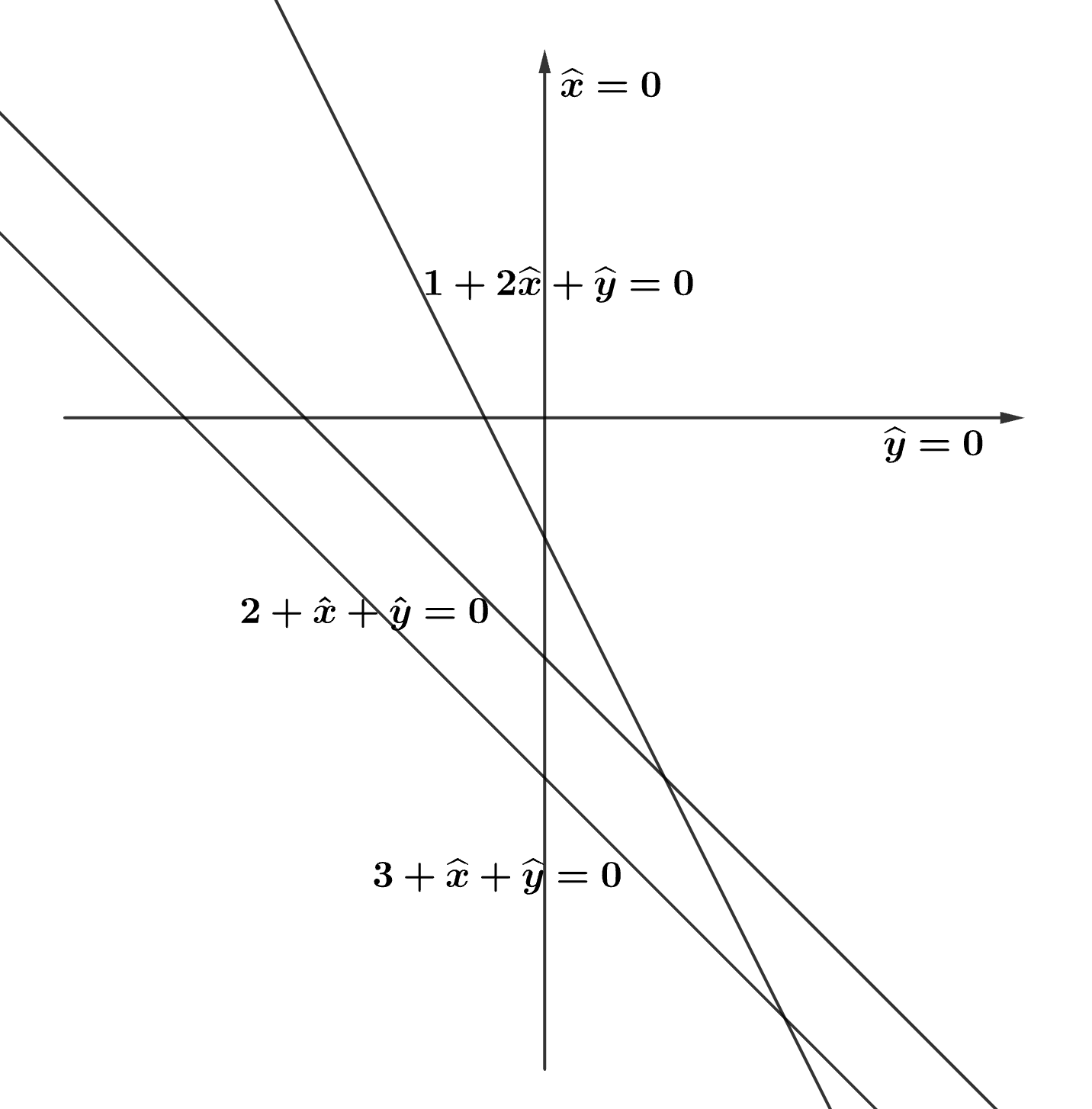}    
        \caption{ML degree 5}
        \end{subfigure}
        \begin{subfigure}{0.5\textwidth}
            \includegraphics[scale=0.25]{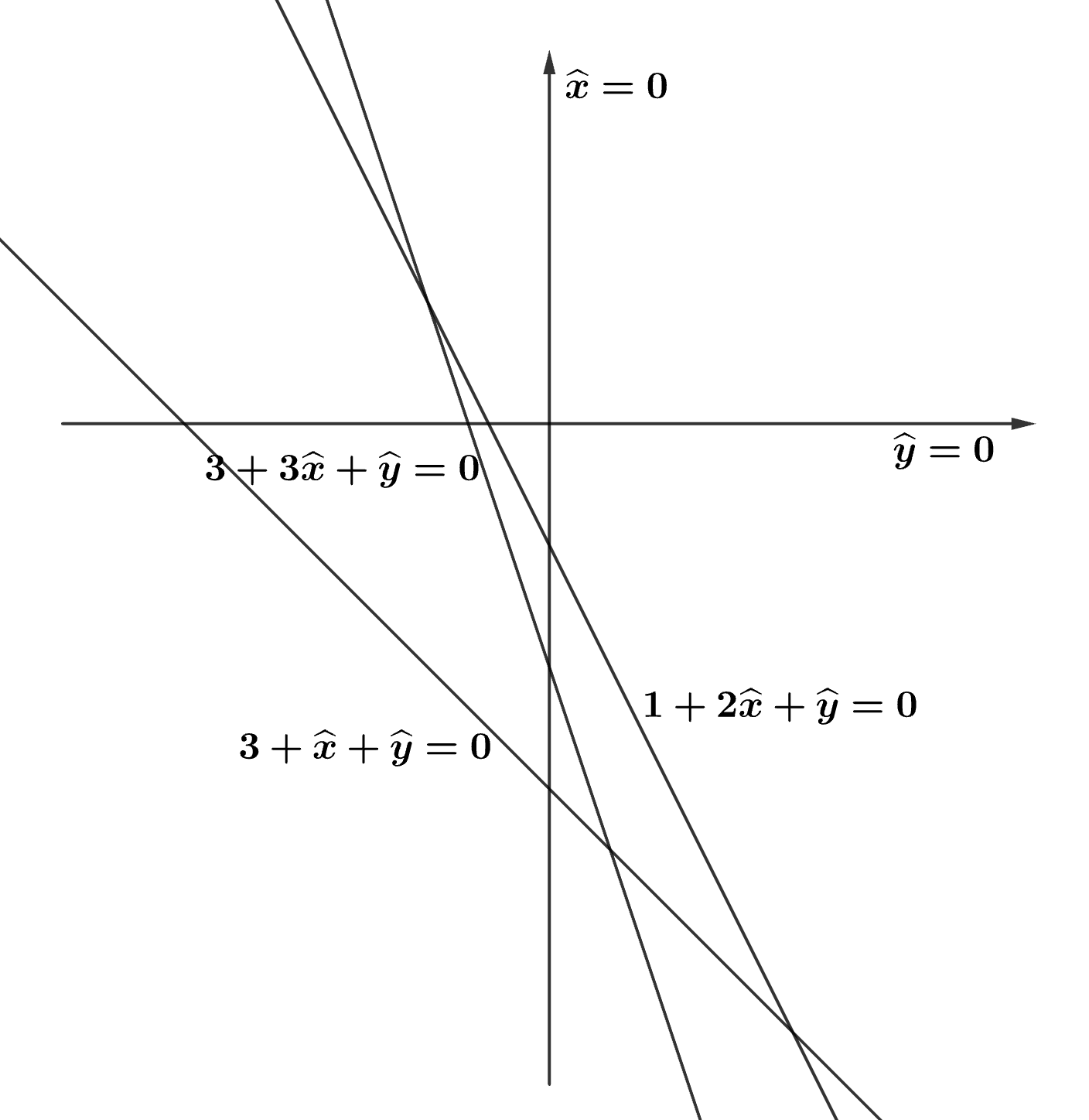}    
        \caption{ML degree 6}
        \end{subfigure}

        \caption{Hyperplane arrangements for 
        $\mathbb{P}^2\times\mathbb{P}^2$ realizing 
        all possible ML degrees}
        \label{fig: line_arrangements}
        \end{figure}
\end{example}

\begin{proposition}\label{prop: P2xPn-1}
    Let $n\geq 3$ and $1\leq k\leq f(n) = \binom{n+1}{2}$. There exist $n$ affine lines $\ell_i=0$ in $\RR^2$ such that 
  \begin{enumerate}  
   \item no line contains $(0, 0)$, 
   \item no line is parallel to $x=0$ or $y=0$, and 
   \item the arrangement $\{\ell_i=0: \, 1 \leq i\leq n\} \cup
   \{x=0\} \cup \{ y=0\}$  forms $k$ bounded regions.
  \end{enumerate}
\end{proposition}
\begin{proof}
    We prove the following stronger statement by induction. There exist $n$ lines $\ell_i=0$ so that together with $\{x=0\}$ and $\{y=0\}$ the arrangement of these $n+2$ lines has $k$ bounded regions where  
    \begin{itemize}
        \item no line $\ell_i=0$ contains $(0, 0)$,
        \item no line $\ell_i=0$ is parallel to $x=0$ or $y=0$,
        \item for $k\geq f(n) - n$, all $\ell_i=0$ are distinct and there exists a point other than $(0,0)$ which is in the intersection of exactly two different lines.
    \end{itemize}
    The base case is $n=3$, $f(3)=6$. The arrangements with 1 and 2 bounded regions use coinciding lines. For $3\leq k\leq 6$ consider the arrangements given in Example~\ref{ex: line arrangements}, satisfying the induction hypothesis, which proves the base case.
    Suppose the statement holds for an arrangement of $n$ lines. Notice that $f(n+1) = f(n) + n+1$. Consider $1\leq k\leq f(n+1)$. If $k\leq f(n)-1 = f(n+1) - n-2$, we can find an arrangement of $n$ lines with $k$ bounded regions by induction hypothesis, and we draw the last line coinciding with one of the existing ones. If $k = f(n+1)$, we use $n+1$ lines in general position. Now suppose $$f(n)\leq k\leq f(n+1)-1 = f(n)+n.$$
    We first construct an arrangement of $n$ distinct lines with $k-n$ bounded regions such that there exist an intersection point other than $(0,0)$ of exactly two lines, $\ell$ and $\ell'$. One of the lines can be $x=0$ or $y=0$. Such an arrangement exists by the induction hypothesis since $k-n\geq f(n)-n$. Now we add a new line $\ell_{n+1}$ going through the special intersection point, so that it also intersects all the other lines in different points. The new line has $n+1$ intersection points. Thus, there are $n$ new bounded regions in the arrangement. As a result we have an arrangement of $n+1$ lines with $k$ bounded regions. By genericity, we may assume that the new line is not parallel to the coordinate axes. Since $n \ge 3$, there exists a line $\ell_r$ distinct from $\ell$ and $\ell'$. By construction, the lines $\ell_{n+1}$ and $\ell_r$ intersect at a point that is not contained on any other line. So we have shown that the inductive step holds for $n+1$ lines.
    \end{proof}
    
    \begin{corollary} \label{cor: P2xPn-1}
        Let $1\leq k\leq \binom{n+1}{2}$. There exists a scaling matrix $w\in (\RR^*)^{3\times n}$ such that the scaled Segre embedding $X_w$ of $\mathbb{P}^2\times \mathbb{P}^{n-1}$ has $\mldeg(X_w) = k$.
    \end{corollary}
    \begin{proof}
        By Corollary~\ref{cor: ML degree by counting bounded regions}, the ML degree equals the number of bounded regions in the following arrangement of $n+2$ lines $x=0$, $y=0$ and $\ell_i = 0$ where
         $\ell_i = w_{1i} + w_{2i}x + w_{i3}y$ for each $i\leq n$.
         The condition that $w$ has non-zero entries is equivalent to no line $\ell_i=0$ being parallel to the axes or passing through the point $(0,0)$. 
         By Proposition \ref{prop: P2xPn-1} there is an arrangement with $k$ bounded regions, which gives the desired $w$.
    \end{proof}

For scaled Segre embeddings of $\mathbb{P}^3\times\mathbb{P}^{n-1}$ the matrix $\widehat{w}$ is $4\times (4+n)$. We move the corresponding hyperplane arrangement to the affine patch $(1, \widehat x, \widehat y, \widehat z)$. The result is an arrangement of $n+3$ planes in $\mathbb{R}^3$, given by $\widehat{w}_{1i} + \widehat{w}_{2i}\widehat x + \widehat{w}_{3i}\widehat y + \widehat{w}_{4i}\widehat z = 0$. The arrangement contains 3 coordinate planes $\widehat x=0$, $\widehat y=0$, $\widehat z=0$. Since $w$ has no zero entries, the planes defined by a column of $w$ intersect each coordinate axis and does not contain the origin.




In the rest of the section, we call a point of intersection of exactly $k$ hyperplanes a $k$-intersection point.
\begin{proposition}
    For any $n\geq 1$ and $1\leq k\leq \binom{n+2}{3}$ there exist $n$ affine planes in $\RR^3$ such that together with the coordinate planes the resulting arrangement has $k$ bounded regions.
\end{proposition}
\begin{proof}
We prove this statement by induction. The cases $n=1$ and $2$ are obvious. Let $n\geq 3$. Denote $f(n) = \binom{n+2}{3}$. If $k\leq f(n-1)$ we draw an arrangement of $n-1$ planes with $k$ regions, which exists by induction hypothesis. The last plane will coincide with one of the existing ones. Define the \textit{deficiency} of the arrangement of $n$ planes to be the difference between $f(n)$ and the actual number of bounded regions in the arrangement. The goal is to determine an algorithm of constructing plane arrangements with deficiencies between $1$ and $$f(n)-f(n-1)-1 = \binom{n+2}{3}-\binom{n+1}{3}-1 = \binom{n+1}{2}-1.$$
This follows from Lemma \ref{lem: deficiency} below and here we explain the reasoning. If $n$ is even, we can write the maximal needed deficiency as
$$\binom{n+1}{2}-1 = \frac{n(n+1)}{2}-1 = n\,\frac{n}{2} + \left(\frac{n}{2}-1\right).$$
An arrangement of $\frac{n}{2}$ pairs of parallel planes with ($\frac{n}{2}-1$) 4-intersection points will have the maximal deficiency $\binom{n+1}{2}-1$. This arrangement will have $f(n-1)+1$ bounded regions. For smaller deficiencies $d\leq \binom{n+1}{2}-1$ write $d = qn + r$, where $q\leq \frac{n}{2}$ and $r<n$ and draw the corresponding arrangement obtained from Lemma \ref{lem: deficiency}.
If $n$ is odd, the maximal deficiency is
$$\binom{n+1}{2}-1 = n\,\frac{n-1}{2} + (n-1).$$
The arrangement of $\frac{n-1}{2}$ pairs of parallel planes with an extra plane in general position  such that there are $(n-1)$ 4-intersection points will have the maximal deficiency. Thus, the number of bounded regions will be $f(n-1)+1$. For smaller deficiencies the same reasoning as above works.
\end{proof}
\begin{lemma}
\label{lem: deficiency}
There exist $n\geq 3$ planes in $\RR^3$ containing $q\leq \frac{n}{2}$ pairs of parallel planes and $r<n$ 4-intersection points, such that the arrangement obtained from these planes together with the three coordinate planes has deficiency $qn+r$.
%
\end{lemma} 
\begin{proof}
     In our proof when we refer to an arrangement of planes we mean the arrangement obtained from these planes together with the three coordinate planes. 
     First we will show the statement for $r=0$. Consider the following algorithm that gives an arrangement of $n$ planes with $q$ pairs of parallel ones.
    \begin{itemize}
        \item Step 0: Draw $n-q \geq \frac{n}{2}$ planes in general position,
        \item Repeat the following step $q$ times: Choose a plane without a parallel pair and draw a parallel one. 
    \end{itemize}
    We may choose the parallel planes so that the arrangement has no 4-intersection points. 
    The goal is to find the difference between the number of bounded regions formed by the arrangement of $n$ planes in general position and the number of bounded regions of the result of the algorithm. To do that, for each $1\leq i\leq q$ we will find the difference between the number of new bounded regions when we add a plane in general position to the arrangement of $(n-q)+i-1$ planes in general position and the number of new regions after step $i$ of the algorithm. Step $i$ draws a new plane which is parallel to one of the existing ones. Consider the lines of intersection of the new plane with existing ones. The arrangement of lines contains $(n-q+2)$ lines in general position and $(i-1)$ lines which are parallel pairs. 
    An arrangement of $(n-q+2)$ lines in general position has $\binom{n-q+1}{2}$ bounded regions. Given an arrangement of $m = (n-q+2+j)$ lines in the plane with $0 \le j < i-1$, the addition of a line that is parallel to exactly one of the $m$ lines introduces exactly $n-q+j$ bounded regions. Hence the number of bounded regions of the line arrangement is
\begin{align*}
    & \binom{n-q+1}{2}+(n-q)+(n-q+1)+\ldots + (n-q+i-2) = \\
    &(n-q) + \sum_{j=1}^{n-q+i-2}j  = - (i-1) + \sum_{j=1}^{n-q+i-1}j  = \binom{n-q+i}{2}-(i-1).
\end{align*}
    The number of new bounded regions produced when we add a plane in general position to the arrangement of $(n-q)+i-1$ planes in general position is
    $$\binom{n-q+i+1}{2}.$$
    Thus the increase of deficiency of step $i$ is 
    $$\binom{n-q+i+1}{2} - \binom{n-q+i}{2}+(i-1) = n-q+i+(i-1),$$
    and the total deficiency is 
    $$\sum_{i=1}^q (n-q+i+(i-1)) = nq-q^2 + \frac{q(q+1)}{2} + \frac{q(q-1)}{2} = nq.$$

    For the case $r > 0$ we modify the arrangement. Let us fix one of $n$ planes and change the rest of the arrangement in the following way to create $r\leq n-1$ 4-intersection points. The fixed plane creates three 3-intersection points with coordinate planes. We repeat the following process $r$ times. Select a $3$-intersection point $p$ of fixed planes and coordinate planes $p = P_1 \cap P_2 \cap P_3$ together with a plane $Q$ that has not been fixed, such that $Q$ is not parallel to any of $P_1$, $P_2$, or $P_3$.
    Translate $Q$ in a direction parallel to its normal so that it passes through $p$, 
    thus creating a 4-intersection point.
    We then add $Q$ to the set of fixed planes and repeat. By the original construction of the arrangement, the planes $P_1$, $P_2$, $P_3$ and $Q$ are in general position or they are coordinate planes. So the translation introduces exactly one $4$-intersection point.
    At each of the $r$ steps, we lose exactly one bounded region. Indeed, the number of regions created by a plane is equal to the number of bounded regions of the arrangement of the intersection lines on a plane. Consider this arrangement for $Q$. 
    Before the translation, the arrangement of lines on $Q$  consists of $n-1$ or $n-2$ lines with no 3-intersections. In particular, the number of bounded regions in $Q$ depends only upon the number of pairs of parallel lines. If we translate $Q$ in a direction parallel to its normal, then we introduce no new lines on $Q$ and parallel lines on $Q$ remain parallel. Therefore, by moving $Q$ we create one 4-intersection point, which reduces the number of bounded regions by exactly one.
    \end{proof}

\begin{corollary} \label{cor: P3xPn-1}
    Let $1\leq k\leq \binom{n+2}{3}$.
    There exists a scaling matrix
    $w\in (\RR^*)^{4\times n}$ such that 
    the scaled Segre embedding $X_w$ of $\PP^3\times\PP^{n-1}$ has $\mldeg(X_w)=k$.
\end{corollary}

This finishes the proof of Theorem \ref{thm: all ML degrees are attained}. The following proposition presents the partial result about the possible ML degrees of scaled Segre embeddings of $\mathbb{P}^{m-1}\times\mathbb{P}^{n-1}$ for $m>4$. It shows that for  big enough $n$ it is possible to find scalings realizing ML degrees ranging from a fixed constant up to the degree of the variety.

\begin{proposition}
Let $d>0$. There exists $n_0$ and $k_0$, such that for every $n\geq n_0$ we can realize an arrangement of $n$ hyperplanes in $\RR^d$ together with $d$ coordinate hyperplanes such that the number of bounded regions is $k$ for all 
    $$k_0 \leq k \leq \binom{n+d-1}{d}.$$
\end{proposition}
\begin{proof}
     Denote the maximum number of bounded regions formed by $n$ hyperplanes and $d$ coordinate hyperplanes by $f_d(n) = \binom{n+d-1}{d}$.  Consider the following algorithm constructing hyperplane arrangements. Fix an integer $0 \le t \le n-1$ and a sequence of positive integers $A_1, \dots, A_t$ such that $A_1 + \dots + A_t \le n-1$.
 \begin{itemize}
     \item Add the first hyperplane in general position,
     \item For $j$ from $1$ to $t$:
     draw $A_j$ hyperplanes through a $d$-intersection point,
     \item Draw $n-1-\sum_{j=1}^tA_j$ hyperplanes in general position.
  \end{itemize}
 For instance, if $t=0$, the result is $n$ hyperplanes in general position, which gives $f_d(n)$ regions. If $t>0$, the arrangement will have fewer bounded regions. The goal now is to estimate the difference.
 
Suppose at step $j-1$ there were $M$ hyperplanes together with $d$ coordinate hyperplanes. Define $L(M, j)$ to be the difference between the number of new regions created if we were to add $A_j$ hyperplanes in general position and the actual number of new regions created after step $j$. To estimate this difference, consider the $(d+A_j)$-intersection point after step $j$. If we move each of $A_j$ added hyperplanes in a random direction, the $(d+A_j)$-intersection point will vanish. These $A_j$ hyperplanes will be in general position, since there are no other constraints. If the changes are small enough, this will not affect the number of bounded regions formed by the existing hyperplanes. The bounded regions that can appear this way are regions formed by $d+A_j$ hyperplanes. Then $L(M,j)$ does not depend on $M$ and 
 $$L(M, j) = L(j) = \binom{A_j+d-1}{d}.$$

 We will show that for big enough $n$, the 
deficiency created by this procedure is $s$ for any $1\leq s\leq f_d(n) - f_d(n-1)-1$. It will follow that starting from some $n_0$ we can construct an arrangement with $n$ hyperplanes with $k$ bounded regions where
$$f_d(n-1)+1\leq k\leq f_d(n).$$
Define $A(k)$ to be the 
minimum $n$ such that for every $i$ from $1$ to $k$ the above algorithm can produce an arrangement of $n$ hyperplanes with deficiency $i$.
Explicitly, $A(k)$ is given by
$$A(k) = 1+\max_{i\leq k}\min_{\{A_j\}}\left\{\sum A_j \Big\vert \sum\binom{A_j+d-1}{d}=i\right\}.$$
For example, $A(1)=2$, since we need at least two hyperplanes to construct an arrangement with deficiency $1$. Moreover, $A(2)$ is also $2$.
Now the goal is to show that there exists $n_0$, such that for all $n\geq n_0$ we have
$$A(f_d(n)-f_d(n-1)-1)\leq n.$$
Firstly, we can show that $A(f_d(n))\leq cn$ for some $c>0$. It holds for every $n>0$ that
$$A(f_d(n+1))\leq n+A(f_d(n+1) - f_d(n)),$$
because given an arrangement with total deficiency $f_d(n+1) - f_d(n)$, we can draw $n$ more hyperplanes through one point to increase this
deficiency by $f_d(n)$. 
Now
$$f_d(n+1) - f_d(n) = \binom{n+d}{d}-\binom{n+d-1}{d}=\binom{n+d-1}{d-1}\leq f_d(m),$$
 for $m\geq \sqrt[d]{d(n+d)^{d-1}}$. Indeed,
 $$\binom{n+d-1}{d-1}\leq \frac{(n+d)^{d-1}}{(d-1)!}\leq \frac{m^d}{d!}\leq \binom{m+d-1}{d}.$$
 Thus with $u(n) = A(f_d(n))$,
 $$u(n+1)\leq n+ u\left(\sqrt[d]{d(n+d)^{d-1}}\right).$$
 It follows that there exist $c>0$ and $n_1$, such that for all $n\geq n_1$, $u(n)\leq cn$.
 Finally, there exist $n_0\geq n_1$, such that for all $n>n_0$
 $$A\left(f_d(n)-f_d(n-1)-1\right)\leq A\left(f_d\left(\sqrt[d]{d(n+d)^{d-1}}\right)\right)\leq c \sqrt[d]{d(n+d)^{d-1}} \leq n.$$
 Now let $k_0 = \binom{n_0+d-2}{d}+1$. Given $n\geq n_0$ and $k\geq k_0$, find $n'\geq n_0$, such that 
 $$\binom{n'+d-2}{d}+1 \leq k \leq \binom{n'+d-1}{d}$$
 and construct the arrangement by the algorithm using $n'$ hyperplanes. Then add $n-n'$ hyperplanes which coincide with existing ones.
\end{proof}

\begin{conjecture}
    Fix $m \ge 4$ and $n \ge m$. For all $k$ such that $1 \le k \le \binom{n+m-2}{m-1}$, there exist $n$ hyperplanes in $\RR^{m-1}$, such that, together with the $m-1$ coordinate hyperplanes, the arrangement of all $m+n-1$ hyperplanes has exactly $k$ bounded regions.
\end{conjecture}

\subsection{Computations for \texorpdfstring{$\PP^3 \times \PP^{n-1}$}{P3xPn-1}}

In this section, we compute ML degrees 
of scaled Segre embeddings of $\PP^3 
\times \PP^{n-1}$ via beta invariants. Recall Definition~\ref{def: matroid of w}, that $M_w$ is the matroid of a scaling $w \in (\CC^*)^{m \times n}$ for $\PP^{m-1} \times \PP^{n-1}$. Not all matroids arise in this way as $M_w$ has the following special property.

\begin{definition}
    We say that a matroid $M$ on $[n]$ is \textit{special} if there exists a basis $B$ of $M$ such that for all $i \in B$ and for all $j \in [n] \backslash B$ we have $(B \backslash \{i\}) \cup \{j\}$ is a basis of $M$. In this case, we say that $B$ is a special basis of $M$.
\end{definition}

\begin{proposition}\label{prop: matroids of segre products are special}
    For any scaling matrix  $w \in (\CC^*)^{m \times n}$ the matroid $M_w$ is special.
\end{proposition}

\begin{proof}
    Let $B = \{1,2,\dots, m\}$. The first $m$ columns of $\widehat w$ is the identity matrix, hence $B$ is a basis of $M_w$. Fix $i \in B$ and $j \in [n+m] \backslash B$. Consider the submatrix $S$ of $\widehat w$ given by the columns indexed by $(B \backslash \{i\}) \cup \{j\}$. We have 
    $\det(S) = (-1)^{i-1} w_{ij} \neq 0$, and
    so $(B \backslash \{i\}) \cup \{j\}$ is a basis for $M_w$. We conclude that $M_w$ is special.
\end{proof}

\begin{remark}
    We note that the following converse to Proposition~\ref{prop: matroids of segre products are special} holds. Suppose that $M$ is a special matroid on $[m+n]$ of rank $m$ that is realizable over $\CC$. Then there exists a scaling $w \in (\CC^*)^{m \times n}$ such that $M_w \cong M$ are isomorphic.
\end{remark}

By Theorem~\ref{thm: ML degree is a matroid invariant}, the ML degree of the scaled Segre embedding $X_w$ of $\PP^3 \times \PP^{n-1}$ coincides with the beta invariant of the special matroid $M_w$. In the following examples, we compute the beta invariants of all special matroids in the cases where $n$ is $4$ and $5$. These computations are based on an online \href{https://www-imai.is.s.u-tokyo.ac.jp/~ymatsu/matroid/}{\tt database of matroids}, which, in turn, is based on \cite{avis1996reverse}. In general, not all special matroids are realizable over $\CC$ and checking realizability is a computationally expensive task. The data-sets used below contain non-realizable matroids. However, since $n$ is small, it is reasonable to assume that the special matroids in the following examples are all realizable over $\CC$.

\begin{example}[$n = 4$]
    Up to isomorphism, there are $940$ matroids of rank $4$ on $8$ elements, of which $568$ are special. The beta invariants of these special matroids are tallied in Table~\ref{tab: rank 4 ground set 8 beta invariants}. In particular, all ML degrees between $1$ and $20$ are attained.
\end{example}

\begin{example}[$n = 5$]
    There are $190\,214$ matroids of rank $4$ on $9$ elements up to isomorphism with $185\,253$ special matroids. Their beta invariants are tallied in Table~\ref{tab: rank 4 ground set 9 beta invariants}. In particular, all possible ML degrees between $1$ and $35$ are attained. We observe that only a small proportion of all matroids are not special and the distribution of the number of matroids with a particular beta invariant is unimodal.
\end{example}

Notice that in both examples above the number of matroids realizing the maximum ML degree $\binom{n+m-2}{m-1}$ is one. The reason is clear if we look at an affine patch of the matroid hyperplane arrangement, because the only possible arrangement with the same number of bounded regions is a set of $n$ hyperplanes in $\RR^{m-1}$ in general position. The same holds for ML degree $\binom{n+m-2}{m-1} - 1$. In this case, the arrangement must have exactly one $(m+1)$-intersection point and, up to isomorphism, there is only one way to achieve this.

\begin{table}[t]
    \centering
    \begin{tabular}{r|cccccccccc}
    \toprule
    Beta invariant &    1 &2 &3  &4  &5  &6  &7  &8  &9  &10 \\
    No. matroids&       1 &6 &10 &16 &17 &26 &27 &33 &29 &47 \\
    \midrule
    Beta invariant &    11 &12 &13 &14 &15 &16 &17 &18 &19 &20 \\
    No. matroids&       59 &74 &84 &67 &40 &20 &7  &3  &1  &1 \\
    \bottomrule
    \end{tabular}
    \caption{The beta invariants of all $568$ special matroids of rank $4$ on $[8]$.}
    \label{tab: rank 4 ground set 8 beta invariants}
\end{table}

\begin{table}[t]
    \centering
    \begin{tabular}{r|cccccccccccc}
    \toprule
    Beta invariant &    1 &2 &3  &4  &5  &6  &7  &8   &9  \\
    No. matroids&       1 &9 &20 &34 &48 &75 &93 &133 &168 \\
    \midrule
    Beta invariant &    10 & 11  &12  &13  &14  &15  &16   &17   &18 \\
    No. matroids&       265 & 361 &486 &636 &760 &845 &1180 &1827 &2881 \\
    \midrule
    Beta invariant &    19   &20 & 21    &22    &23    &24    &25    &26    &27 \\
    No. matroids&       4767 &7807 & 11600 &17153 &25328 &33480 &33963 &24293 &11856 \\
    \midrule
    Beta invariant &    28   &29  &30 & 31 &32 &33 &34 &35 \\
    No. matroids&       3961 &967 &199 & 42 &10 &3  &1  &1  \\
    \bottomrule
    \end{tabular}
    \caption{The beta invariants of all $185\,253$ special matroids of rank $4$ on $[9]$.}
    \label{tab: rank 4 ground set 9 beta invariants}
\end{table}

\begin{remark}
    Without a significant theoretical improvement, it does not seem possible to compute the beta invariants of all special matroids in the case where $n = 6$. There are $4\,886\,380\,924$ matroids of rank $4$ on $10$ elements and, similarly to the $n = 5$ example, we expect that a large proportion of these are special.
\end{remark}

\subsection{The principal \texorpdfstring{$A$}{A}-determinant of \texorpdfstring{$\PP^1 \times \PP^1 \times \PP^{n-1}$}{P1xP1xPn-1}}

In this section we determine the principal $A$-determinant of the Segre embedding $\PP^1 \times \PP^1 \times \PP^{n-1}$ and investigate the ML degrees under certain scalings. 
The polytope $P_n \subset \RR^{2} \oplus \RR^{2} \oplus \RR^{n}$ of $\PP^1 \times \PP^1 \times \PP^{n-1}$ is the convex hull of $e_i \oplus e_j \oplus f_k$ for each $i, j \in [2]$ and $k \in [n]$. Here $e_i$ and $f_k$ are the standard unit
vectors of $\RR^2$ and $\RR^n$, respectively.
The polytope $P_n$ is the product of standard simplices $\Delta_1 \times \Delta_1 \times \Delta_{n-1}$ and it is the set of points $x \oplus y \oplus z$ defined by the following equations and inequalities:
\[
x_1 + x_2 = 1, \, 0 \le x_i \le 1, \quad y_1 + y_2 = 1, \,  0 \le y_j \le 1,  \quad z_1 + \dots + z_n = 1, \, 0 \le z_k \le 1
\]
for all $i, j \in [2]$ and $k \in [n]$. 

\begin{proposition}\label{prop: 2x2xn faces of Pn}
    Up to an affine unimodular transformation, $P_n = \Delta_1 \times \Delta_1 \times \Delta_{n-1}$ has two kinds of non-simplicial proper faces. There are $2\binom{n}{k}$ faces equal to $P_k$ for each $k \in [n-1]$. All others are equal to $\Delta_1 \times \Delta_k$ for some $k \in [n-1]$.
\end{proposition}
    
\begin{proof}
    This follows from the formulation of $P_n$ by inequalities. The facets of $P_n$ are supported by the hyperplanes defined by $x_i = 0$, $y_i = 0$, and $z_j = 0$ for each $i \in [2]$ and $j \in [n]$. If the facet is defined by $z_j = 0$, then it is equal to $P_{n-1}$. Otherwise, if the facet is defined by either $x_i = 0$ or $y_i = 0$, then it is equal to $\Delta_1 \times \Delta_{n-1}$.
\end{proof}

The only faces of $P_n$ for which we require a novel description of the $A$-discriminant are those faces of the form $P_k$ for each $k \in [n]$. In this case, the $A$-discriminant is known as a \textit{hyperdeterminant of format} $2 \times 2 \times k$. For further details see \cite[Chapter~14]{Gelfand1994}.

\begin{proposition}[{\cite[Chapter~14, Theorem~1.3]{Gelfand1994}}]\label{prop: 2x2xn ADiscriminant}
    The $A$-discriminant for $P_n$ is non-trivial if and only if $n = 2$ or $n = 3$. In these cases, the $A$-discriminant is the hyperdeterminant for the $2\times 2 \times 2$ and $2 \times 2 \times 3$ tensors, respectively.
\end{proposition}

In the following examples, we give an explicit description of these hyperdeterminants.

\begin{example}
    The $2 \times 2 \times 2$ hyperdeterminant, often called \textit{Cayley's hyperdeterminant} after the approach taken in \cite{cayley1845theory}, is the $A$-discriminant associated to the matrix
    \[
    A = \begin{pmatrix}
      1&1&1&1&0&0&0&0\\
      1&1&0&0&1&1&0&0\\
      1&0&1&0&1&0&1&0\\
      0&1&0&1&0&1&0&1
    \end{pmatrix},
    \]
    which is
    \begin{align*}
    &w_{4}^{2}w_{5}^{2}+w_{3}^{2}w_{6}^{2}+w_{2}^{2}w_{7}^{2}+w_{1}^{2}w_{8}^{2}-2w_{3}w_{4}w_{5}w_{6}-2w_{2}w_{4}w_{5}w_{7}-2w_{2}w_{3}w_{6}w_{7}\\
    &-2w_{1}w_{4}w_{5}w_{8}-2w_{1}w_{3}w_{6}w_{8}-2w_{1}w_{2}w_{7}w_{8}+4w_{1}w_{4}w_{6}w_{7}+4w_{2}w_{3}w_{5}w_{8}.
    \end{align*}
\end{example}

\begin{example}
    The $2 \times 2 \times 3$ hyperdeterminant is the $A$-discriminant associated to the matrix
    \[
    \setcounter{MaxMatrixCols}{12}
    A = \begin{pmatrix}
      1&1&1&1&0&0&0&0&0&0&0&0\\
      0&0&0&0&1&1&1&1&0&0&0&0\\
      1&1&0&0&1&1&0&0&1&1&0&0\\
      1&0&1&0&1&0&1&0&1&0&1&0\\
      0&1&0&1&0&1&0&1&0&1&0&1
    \end{pmatrix},
    \]
    which is given by the polynomial 
    \begin{align*}
        &w_{4}^{2}w_{6}w_{7}w_{9}^{2}-w_{3}w_{4}w_{6}w_{8}w_{9}^{2}-w_{2}w_{4}w_{7}w_{8}w_{9}^{2}
        +w_{2}w_{3}w_{8}^{2}w_{9}^{2}-w_{4}^{2}w_{5}w_{7}w_{9}w_{10}\\
        &-w_{3}w_{4}w_{6}w_{7}w_{9}w_{10}+w_{2}w_{4}w_{7}^{2}w_{9}w_{10}
        +w_{3}w_{4}w_{5}w_{8}w_{9}w_{10}+w_{3}^{2}w_{6}w_{8}w_{9}w_{10} \\
        &-w_{2}w_{3}w_{7}w_{8}w_{9}w_{10}+w_{1}w_{4}w_{7}w_{8}w_{9}w_{10}-w_{1}w_{3}w_{8}^{2}w_{9}w_{10}
        +w_{3}w_{4}w_{5}w_{7}w_{10}^{2} \\
        &-w_{1}w_{4}w_{7}^{2}w_{10}^{2}-w_{3}^{2}w_{5}w_{8}w_{10}^{2}+w_{1}w_{3}w_{7}w_{8}w_{10}^{2}-w_{4}^{2}w_{5}w_{6}w_{9}w_{11}+w_{3}w_{4}w_{6}^{2}w_{9}w_{11}\\
        &-w_{2}w_{4}w_{6}w_{7}w_{9}w_{11}+w_{2}w_{4}w_{5}w_{8}w_{9}w_{11}-w_{2}w_{3}w_{6}w_{8}w_{9}w_{11}+w_{1}w_{4}w_{6}w_{8}w_{9}w_{11} \\
        &+w_{2}^{2}w_{7}w_{8}w_{9}w_{11}-w_{1}w_{2}w_{8}^{2}w_{9}w_{11}
        +w_{4}^{2}w_{5}^{2}w_{10}w_{11}-w_{3}w_{4}w_{5}w_{6}w_{10}w_{11}\\
        &-w_{2}w_{4}w_{5}w_{7}w_{10}w_{11}+2w_{1}w_{4}w_{6}w_{7}w_{10}w_{11}
        +2w_{2}w_{3}w_{5}w_{8}w_{10}w_{11}-2w_{1}w_{4}w_{5}w_{8}w_{10}w_{11}\\
        &-w_{1}w_{3}w_{6}w_{8}w_{10}w_{11}-w_{1}w_{2}w_{7}w_{8}w_{10}w_{11}
        +w_{1}^{2}w_{8}^{2}w_{10}w_{11}+w_{2}w_{4}w_{5}w_{6}w_{11}^{2}\\
        &-w_{1}w_{4}w_{6}^{2}w_{11}^{2}-w_{2}^{2}w_{5}w_{8}w_{11}^{2} +w_{1}w_{2}w_{6}w_{8}w_{11}^{2}+w_{3}w_{4}w_{5}w_{6}w_{9}w_{12}
        -w_{3}^{2}w_{6}^{2}w_{9}w_{12} \\
        &+w_{2}w_{4}w_{5}w_{7}w_{9}w_{12}+2w_{2}w_{3}w_{6}w_{7}w_{9}w_{12}
        -2w_{1}w_{4}w_{6}w_{7}w_{9}w_{12}-w_{2}^{2}w_{7}^{2}w_{9}w_{12}\\
        &-2w_{2}w_{3}w_{5}w_{8}w_{9}w_{12}+w_{1}w_{3}w_{6}w_{8}w_{9}w_{12}
        +w_{1}w_{2}w_{7}w_{8}w_{9}w_{12}-w_{3}w_{4}w_{5}^{2}w_{10}w_{12}\\
        &+w_{3}^{2}w_{5}w_{6}w_{10}w_{12}-w_{2}w_{3}w_{5}w_{7}w_{10}w_{12}
        +w_{1}w_{4}w_{5}w_{7}w_{10}w_{12}-w_{1}w_{3}w_{6}w_{7}w_{10}w_{12}\\
        &+w_{1}w_{2}w_{7}^{2}w_{10}w_{12}+w_{1}w_{3}w_{5}w_{8}w_{10}w_{12}
        -w_{1}^{2}w_{7}w_{8}w_{10}w_{12}-w_{2}w_{4}w_{5}^{2}w_{11}w_{12}\\
        &-w_{2}w_{3}w_{5}w_{6}w_{11}w_{12}+w_{1}w_{4}w_{5}w_{6}w_{11}w_{12}
        +w_{1}w_{3}w_{6}^{2}w_{11}w_{12}+w_{2}^{2}w_{5}w_{7}w_{11}w_{12}\\
        &-w_{1}w_{2}w_{6}w_{7}w_{11}w_{12}+w_{1}w_{2}w_{5}w_{8}w_{11}w_{12}
        -w_{1}^{2}w_{6}w_{8}w_{11}w_{12}+w_{2}w_{3}w_{5}^{2}w_{12}^{2}\\
        &-w_{1}w_{3}w_{5}w_{6}w_{12}^{2}-w_{1}w_{2}w_{5}w_{7}w_{12}^{2}
        +w_{1}^{2}w_{6}w_{7}w_{12}^{2}.
    \end{align*}
\end{example}

\begin{proof}[Proof of Theorem~\ref{thm:principal A-det P1xP1xPn}]
Proposition~\ref{prop: 2x2xn faces of Pn} classifies all non-simplical faces of $\Delta_1 \times \Delta_1 \times \Delta_{n-1}$. The factors of the principal $A$-determinant come from these faces. Out of the two types of faces identified in the same proposition, the first type contributes either
$2 \times 2 \times 2$ hyperdeterminants
or $2 \times 2 \times 3$ hyperdeterminants 
of the $2 \times 2 \times n$ tensor $w$ by 
Proposition~\ref{prop: 2x2xn ADiscriminant}
The second type of face contributes the $2$-minors of all slices of $w$.
\end{proof}

Recall from Theorem~\ref{thm: ML degree is a matroid invariant} that the ML degree of $\PP^{m-1} \times \PP^{n-1}$ with respect to a scaling $w$ is determined precisely by which factors of the principal $A$-determinant vanish. We conjecture that the same is true for $\PP^1\times\PP^1 \times \PP^{n-1}$.

\begin{conjecture}
    Let $w$ and $w'$ be scalings for $\PP^1 \times \PP^1 \times \PP^n$ such that for any face $\Gamma \subseteq P_n$ we have $\Delta_\Gamma(w) = 0$ if and only if $\Delta_\Gamma(w') = 0$. Then $\mldeg(X_w) = \mldeg(X_{w'})$.
\end{conjecture}

Also, as in the case of $\PP^{m-1} \times \PP^{n-1}$, the vanishing of certain factors 
of the principal $A$-determinant $E_A$ forces
other factors to vanish as well. Below we give two such cases. Then we return to the question of the possible ML degrees attained as $w$ varies. To keep track of the determinants and hyperdeterminants of $w$, we introduce the following notation:  
For each subset $K \in \binom{[n]}{2}$ or
$K \in \binom{[n]}{3}$, we write $D_K$ for the $2 \times 2 \times 2$ and $2 \times 2 \times 3$ hyperdeterminant of $(w_{ijk})_{k \in K}$.

\begin{definition}
    Let $U = \{U_1, U_2, U_3\}$ be a set of three distinct $2$-minors of slices of $w$. We say that $U$ is a \textit{square cup} if each of the following hold:
    \begin{itemize}
        \item There is a $2 \times 2 \times 2$ sub-tensor of $w$ containing $U_1$, $U_2$, and $U_3$,
        \item The minors $U_1$ and $U_2$ have disjoint sets of variables.
    \end{itemize}
\end{definition}

\begin{proposition} \label{prop:hyperdet vanishes 1}
    Let $w = (w_{ijk}) \in (\CC^*)^{2 \times 2 \times 2}$. If there exists a square cup $U$ such that each $2$-minor of $U$ vanishes on $w$, then the hyperdeterminant of $w$ vanishes.
\end{proposition}
\begin{proof}
    Let $I = \langle U_1, U_2, U_3 \rangle \subseteq \CC[w_{ijk}]$ be the ideal generated by the square cup. Let $J = (I : (\prod_{i,j,k} w_{ijk})^\infty)$ be the saturation of $I$ by the product of variables. By a straightforward computation in {\tt Macaulay2} \cite{M2}, the ideal $J$ is minimally generated by all six $2$-minors of $w$ and contains the hyperdeterminant of $w$.
\end{proof}

\begin{proposition} \label{prop:hyperdet vanishes 2}
    Let $w = (w_{ijk}) \in (\CC^*)^{2 \times 2 \times 3}$. If one of the hyperdeterminants $D_{12}$, $D_{13}$, or $D_{23}$ vanishes then so does the hyperdeterminant $D_{123}$. Also, if there exists a pair of vanishing $2$-minors contained in a $1 \times 2 \times 3$ slice of $w$, then the hyperdeterminant $D_{123}$ vanishes.
\end{proposition}

\begin{proof}
    The proof follows from a direct computation of the saturation of the ideal generated by a hyperdeterminant or the ideal generated by the $2$-minors with respect to the product of the variables.
\end{proof}

\begin{example}
    We show that for the scaled Segre embedding of $\PP^1 \times \PP^1 \times \PP^2$ all possible ML degrees $1,2,\dots, 12$ are achieved. Let $A$ be the matrix defining $X$, which is given by
    \[
    A = \begin{pmatrix}
      1&1&1&1&0&0&0&0&0&0&0&0\\
      0&0&0&0&1&1&1&1&0&0&0&0\\
      1&1&0&0&1&1&0&0&1&1&0&0\\
      1&0&1&0&1&0&1&0&1&0&1&0\\
      0&1&0&1&0&1&0&1&0&1&0&1
      \end{pmatrix}.
    \]
    Table \ref{table: ML degrees of P1 times P1 times P2} shows an example of a scaling of $X_A$ that achieves each possible ML degree.
\end{example}

\begin{table}
    \centering
    \begin{tabular}{cc}
        \toprule
        Scaling $w$ & ML degree \\
        \midrule
        $(1,1,1,1,\  1,1,1,1,\  1,1,1,1)$ & $1$ \\
        $(1,1,1,1,\  1,1,1,1,\  1,1,2,2)$ & $2$ \\
        $(1,1,1,1,\  1,1,1,1,\  1,1,1,2)$ & $3$ \\
        $(1,1,1,1,\  1,1,1,1,\  1,1,2,3)$ & $4$ \\
        $(1,1,1,1,\  1,1,1,1,\  1,2,2,1)$ & $5$ \\
        $(1,1,1,1,\  1,1,1,2,\  1,2,1,1)$ & $6$ \\
        $(1,1,1,1,\  1,1,1,2,\  1,2,3,4)$ & $7$ \\
        $(1,1,1,1,\  1,1,1,2,\  1,2,2,3)$ & $8$ \\
        $(1,1,1,1,\  1,1,1,2,\  1,2,2,5)$ & $9$ \\
        $(1,1,1,2,\  1,2,1,3,\  1,3,2,5)$ & $10$ \\
        $(1,1,1,2,\  1,3,2,5,\  2,3,6,9)$ & $11$ \\
        $(1,1,1,2,\  1,3,2,5,\  1,6,3,7)$ & $12$ \\
        \bottomrule
    \end{tabular}
    \caption{ML degrees of the Segre embedding of $\PP^1 \times \PP^1 \times \PP^2$ for different scalings.} \label{table: ML degrees of P1 times P1 times P2}
\end{table}

The following result gives a large family of scalings $w$ for which the corresponding variety has ML degree one.

\begin{proposition}\label{prop: P1xP1xPn ml degree one}
    Let $X_w$ be a scaled Segre embedding of $\PP^1 \times \PP^1 \times \PP^{n-1}$ with the scaling $w \in (\CC^*)^{2\times 2 \times n}$. If there is a partition of $w = (w_{ijk})$ into parallel slices such that $w$ is constant on each slice, then $X_{w} = X_A$. In particular, $X_{A,w}$ has ML degree one. 
\end{proposition}

\begin{proof}
    The defining ideal $I_A$ of $X_A$ is generated by the binomials 
    \[
    x_{ijk} x_{i'j'k'} - x_{ijk'} x_{i'j'k}, \quad x_{ijk} x_{i'j'k'} - x_{ij'k} x_{i'jk'}, \quad
    x_{ijk} x_{i'j'k'} - x_{ij'k'} x_{i'jk}
    \]
    for all $i,j,i',j' \in [2]$ and $k,k' \in [n]$. Fix a scaling $w$ such that parallel slices of $w$ are constant.
    Notice that each binomial above is contained in at most two parallel slices of the tensor $X = (x_{ijk})$. If the binomial is contained in a single slice where $w$ is constant then the same generator appears in the analogous generating set for $I_{A,w}$. Suppose that a binomial is contained in two parallel slices. The corresponding generator of the ideal $I_{A,w}$ is obtained by scaling each variable by its corresponding entry in $w$. Since $w$ is constant on each slice, each term of the binomial is scaled by the same factor. Hence the binomial itself is a generator of the ideal $I_{A,w}$. We conclude that the ideals $I_A$ and $I_{A,w}$ have the same generating set.
\end{proof}

%
%

We finish this section by proving that the small-valued ML degrees can always be obtained by an extension of 
Theorem \ref{thm: all ML degrees are attained}
for $\PP^{m-1} \times \PP^{n-1}$. 

\begin{proposition} \label{prop:extension}
Let $m$ be equal to $2,3$ or $4$. Then for each $1 \leq k \leq \binom{n+m-2}{m-1}$ there exists a scaling matrix $w \in (\RR^*)^{l \times m \times n}$ such that $\mldeg(X_w) = k$ for the scaled Segre embedding $X_w$ of $\PP^{l-1} \times \PP^{m-1} \times \PP^{n-1}$.
\end{proposition}

\begin{proof}
According to Theorem \ref{thm: all ML degrees are attained}, for $1 \leq k \leq \binom{n+m-2}{m-1}$ there exists a scaling matrix $w' = (w_{st}') \in (\RR^*)^{m \times n}$ such that $\mldeg(X_{w'}) = k$, where $X_{w'}$ denotes the scaled Segre embedding of $\PP^{m-1} \times \PP^{n-1}$. Now consider $X_w$ with the repeating scaling $w_{ist} = w_{st}'$ for all $i \in [l]$. Then $X_w$ is the toric fiber product of $\PP^{l-1}$ with $X_{w'}$ since
the chosen scaling 
$w_{ist}$ can be written as a product $1 \cdot w_{st}$. In this case, the ML degree is multiplicative, see e.g. \cite[Theorem 5.5]{MaximumLikelihoodEstimationOfToricFanoVarieties} and \cite[Section 5]{amendola2023likelihood}. Therefore,
\[
\mldeg (X_w) = \mldeg (\PP^{l-1}) \cdot \mldeg (X_{w'}) = 1 \cdot k = k. \qedhere
\]
\end{proof}

\begin{corollary}
    Fix $d \le n$. The scaled Segre embedding of  $\PP^1 \times \PP^1 \times \PP^{n-1}$ has ML degree $d$ where the scaling $w = (w_{ijk})$ is given by
    \[
    w_{ijk} = \begin{cases}
        1 & \text{if } k > d \text{ or } i = 1, \\
        k & \text{if } k \le d \text{ and } i = 2.
    \end{cases}
    \]
\end{corollary}

\section{Second Hypersimplex}


The focus of this section is on \emph{uniform matroids} of rank two and their ML degree stratification. The \emph{uniform matroid of rank} $k$, denoted $U_d^k$, is the matroid with bases $\binom{[d]}{k}$. Its matroid polytope is the \textit{hypersimplex} $\Delta_{k,d}$, a $(d-1)$-dimensional polytope with $\binom{d}{k}$ vertices embedded in the $d$-dimensional space.

\begin{example}[Octahedron] \label{example: octahedron}
The uniform matroid $U_4^2$ of rank two on four elements has bases
\begin{equation*}
12, \,13, \,14, \,23, \,24, \,34.
\end{equation*}
Its matroid polytope $\Delta_{2,4}$ shown in Figure \ref{fig: octahedron} is the convex hull of the columns of 
\begin{equation*}
A = \begin{pmatrix}
1 & 1 & 1 & 0 & 0 & 0 \\
1 & 0 & 0 & 1 & 1 & 0 \\
0 & 1 & 0 & 1 & 0 & 1 \\
0 & 0 & 1 & 0 & 1 & 1 \\
\end{pmatrix}.
\end{equation*}
\end{example}

\begin{figure}[t]
\centering
\begin{tikzpicture}
\node[draw, circle, inner sep=1.5pt, fill, black] at (2.5, 0) {};
\node[draw, circle, inner sep=1.5pt, fill, black] at (0, 2.5) {};
\node[draw, circle, inner sep=1.5pt, fill, black] at (2, 1.95) {};
\node[draw, circle, inner sep=1.5pt, fill, black] at (3, 3.05) {};
\node[draw, circle, inner sep=1.5pt, fill, black] at (5, 2.5) {};
\node[draw, circle, inner sep=1.5pt, fill, black] at (2.5, 5) {};

\draw[very thin] (2.5,0) -- (0, 2.5) {};
\draw[very thin] (2.5,0) -- (2, 1.95) {};
\draw[very thin] (2.5,0) -- (3, 3.05) {};
\draw[very thin] (2.5,0) -- (5, 2.5) {};

\draw[very thin] (2.5,5) -- (0, 2.5) {};
\draw[very thin] (2.5,5) -- (2, 1.95) {};
\draw[very thin] (2.5,5) -- (3, 3.05) {};
\draw[very thin] (2.5,5) -- (5, 2.5) {};

\draw[very thin] (0, 2.5) -- (2, 1.95) {};
\draw[very thin] (0, 2.5) -- (3, 3.05) {};
\draw[very thin] (5, 2.5) -- (3, 3.05) {};
\draw[very thin] (5, 2.5) -- (2, 1.95) {};
\end{tikzpicture}
\caption{Matroid polytope of $U_4^2$.} \label{fig: octahedron}
\end{figure}
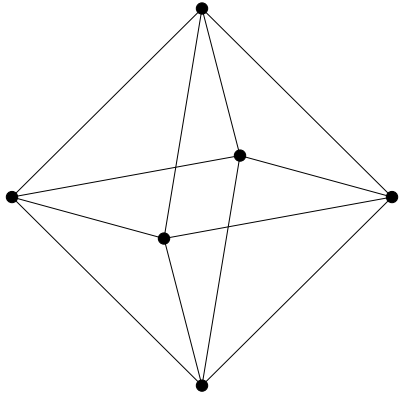

The ML degree of the scaled toric variety $X_{2,d}^w$ corresponding to the second hypersimplex $\Delta_{2,d}$ is bounded 
above by the degree of the same variety. 

\begin{proposition} \label{prop:matroid-degree}
Let $X_{2,d}$ be the toric variety associated to the second hypersimplex $\Delta_{2,d}$. Then
\begin{equation*}
\textup{deg}(X_{2,d}) = 2^{d-1}-d.
\end{equation*}
\end{proposition}

\begin{proof}
The normalized volume of the hypersimplex $\Delta_{k,d}$ is given by the Eulerian number $A(d-1,k-1)$ \cite{laplace1886oeuvres}. It is well known that $A(d-1, 1) = 2^{d-1} - d$.
\end{proof}

For a better understanding of the achievable ML degrees of $X_{2,d}^w$, we present a compact description of the principal $A$-determinant. The lattice points contained in $\Delta_{2,d}$ are $e_i + e_j$ for $1 \leq i  < j \leq d$ where $e_i$ is the $i$th standard 
unit vector in $\RR^d$. These lattice points 
comprise the columns of the $d \times \binom{d}{2}$ matrix $A$. 
We assign the scaling constant $w_{ij}$ to each of these lattice points. With this, let 
$W$ be the following $d \times d$ symmetric
matrix with entries 
\begin{equation*}
W_{ij} = \begin{cases}
w_{ij}, & i < j, \\
0, & i = j.
\end{cases}
\end{equation*}

\begin{proposition} \label{prop: A-disc for U2n}
The $A$-discriminant $\Delta(U_d^2)$ of 
the second hypersimplex $\Delta_{2,d}$ is
$\det(W)$.
\end{proposition}

\begin{proof}
Let $f = \sum_{i < j} w_{ij}\theta_i \theta_j$
be the polynomial defined by $\Delta_{2,d}$ 
and the scaling $w$.
Since $f$ is homogeneous of degree two, the partial derivatives give a system of linear equations
\begin{equation*}
\frac{\partial f}{\partial \theta_k} = \sum_{i =1}^{k-1} w_{ik} \theta_i + \sum_{j = k+1}^d w_{kj} \theta_j = 0, \quad k \in [d],
\end{equation*}
which is precisely $W \theta = 0$. The system has a non-zero solution if and only if $\det(W) = 0$. 
Hence, $\Delta(U^2_d) = \det(W)$.
\end{proof}

\begin{corollary}
The $A$-discriminant $\Delta(U^2_d)$ is homogeneous with respect to the multigrading given by the columns of the matrix $A$ with degree $(2,2,\dots,2)^T$.
\end{corollary}

\begin{proof}
Every term of the determinant $\det(W)$ is of the form $m_\sigma = W_{1 \sigma(1)} W_{2 \sigma(2)} \dots W_{d \sigma(d)}$ for some permutation $\sigma \in \Sym(d)$. Each non-zero entry $W_{ij}$ is a variable with degree $e_i + e_j$. Therefore, $\textup{deg} (m_\sigma) = \sum_i (e_i + e_{\sigma(i)}) = 2 \sum_i e_i$.
\end{proof}

\begin{example}
Consider the uniform matroid $U_5^2$. According to Proposition \ref{prop: A-disc for U2n}, the $A$-discriminant is 
\begin{equation*}
\Delta(U^2_5) = \det \begin{pmatrix}
0 & w_{12} & w_{13} & w_{14} & w_{15} \\
w_{12} & 0 & w_{23} & w_{24} & w_{25} \\
w_{13} & w_{23} & 0 & w_{34} & w_{35} \\
w_{14} & w_{24} & w_{34} & 0 & w_{45} \\
w_{15} & w_{25} & w_{35} & w_{45} & 0 \\
\end{pmatrix}
\end{equation*}
\vspace*{-10pt}
\begin{align*}
&=w_{14} w_{15} w_{23} w_{25} w_{34} + w_{13} w_{15} w_{24} w_{25} w_{34} + w_{14} w_{15} w_{23} w_{24} w_{35} + w_{13} w_{14} w_{24} w_{25} w_{35} \\
&+ w_{12} w_{15} w_{24} w_{34} w_{35} + w_{12} w_{14} w_{25} w_{34} w_{35} + w_{13} w_{15} w_{23} w_{24} w_{45} + w_{13} w_{14} w_{23} w_{25} w_{45} \\
&+ w_{12} w_{15} w_{23} w_{34} w_{45} + w_{12} w_{13} w_{25} w_{34} w_{45} + w_{12} w_{14} w_{23} w_{35} w_{45} + w_{12} w_{13} w_{24} w_{35} w_{45} \\ 
&- w_{15}^2 w_{23} w_{24} w_{34} - w_{13} w_{14} w_{25}^2 w_{34} - w_{12} w_{15} w_{25} w_{34}^2 - w_{13} w_{15} w_{24}^2 w_{35} -  w_{14}^2 w_{23} w_{25} w_{35} \\
&- w_{12} w_{14} w_{24} w_{35}^2 - w_{14} w_{15} w_{23}^2 w_{45} - w_{13}^2 w_{24} w_{25} w_{45} - w_{12}^2 w_{34} w_{35} w_{45} - w_{12} w_{13} w_{23} w_{45}^2.
\end{align*}
\end{example}

We are now ready to present the principal $A$-determinant of the second hypersimplex. For 
every subset $E \subset [d]$ with $|E|\geq 2$, we denote the uniform matroid 
of rank two on the ground set $E$ by $U_E^2$.

\begin{theorem}\label{thm: A-det for U2n}
The principal $A$-determinant of the
second hypersimplex $\Delta_{2,d}$ for $d \ge 4$ is
\[
\prod_{\substack{E \subseteq [d] \\ |E| \ge 4}} \Delta(U^2_E).
\]
In particular, the only faces of $\Delta_{2,d}$ that contribute non-unit factors to this principal $A$-determinant  arise from matroid deletion.
\end{theorem}

\begin{proof}
The hypersimplex $\Delta_{2,d}$ has the half-space description
\[
\Delta_{2,d} = \{x \in \RR^d \colon 0 \le x_i \le 1 \text{ for all } i \in [d] \text{ and } x_1 + x_2 + \dots + x_d = 2 \}.
\]
Therefore, the facets of $\Delta_{2,d}$ are defined by the hyperplanes $x_i = 0$ and $x_i = 1$ for each $i \in [d]$. In particular, $\Delta_{2,d}$ has $d$ facets of the form $\Delta_{2,d} \cap \{ x_i = 0\}$, which are unimodularly equivalent to $\Delta_{2,d-1}$.
These facets are the matroid polytopes of $U^2_E$ with $E = [d] \backslash \{i\}$. The other $d$ facets of the form $\Delta_{2,d} \cap \{ x_i = 1\}$ are unimodular simplices of dimension $d-1$ (see e.g. \cite{Borcea2008InfinitesimallyFS}).

By induction on $d$, the faces of $\Delta_{2,d}$ are the matroid polytopes of $U^2_E$ with $E \subseteq [d]$ or unimodular simplices. Each simplicial face contributes a trivial factor to the principal $A$-determinant. In particular, the matroid polytopes of $U^2_2$ and $U^2_3$ are unimodular simplices. So the non-simplicial faces of $\Delta_{2,d}$ are the matroid polytopes of $U^2_E$ for each $E \subseteq [d]$ with $\vert E \vert \ge 4$.
\end{proof}

As in the previous section we want to understand which ML degrees can be achieved by choosing different scalings. The singularities of the polynomials defined by the faces of the underlying polytope can help with this goal (see e.g. \cite[Example 4.4]{MaximumLikelihoodEstimationOfToricFanoVarieties}).

\begin{proposition} \label{prop:3DFace}
Let $\Gamma_3$ be a three-dimensional face of $\Delta_{2,d}$. Then $f_w$ defined by $\Gamma_3$ has at most one singularity.
\end{proposition}

\begin{proof}
We take $\Gamma_3$ to be $\Delta_{2,4}$, 
and after an affine unimodular transformation we can remove the first coordinate of each vertex of $\Delta_{2,4}$ to get a full-dimensional
polytope in $\RR^3$.
Each lattice point of this polytope has at most two non-zero entries. Therefore, 
\begin{equation*}
\frac{\partial f_w}{\partial \theta_i}(\theta) = 0, \quad 1 \leq i \leq 3,
\end{equation*}
can be written as a system of linear equations $W \theta = b$. 
Since $\textup{det} (W) = 2 w_{12} w_{13} w_{23} \neq 0$, there is a unique solution for $\theta$ in $w_{ij}$. Substitution of $\theta$ in $f_w$ yields $f_w=0$ or $f_w \neq 0$.
\end{proof}

\begin{corollary} \label{cor: singularity}
 Let $w$ be a scaling such that $k$ factors of the principal $A$-determinant of the second hypersimplex $\Delta_{2,d}$ corresponding to the principal minors of $W$
 of size four vanish. Then $\mldeg(X_{2,d}^w) \leq \deg(X_{2,d}) - k$. In particular, the ML degree of $X_{2,4}^w$ is either $3$ or $4$.
\end{corollary}

\begin{proof}
A principal minor of $W$ of size four
corresponds to a face of $\Delta_{2,d}$ given by $\Delta_{2,4}$. If it vanishes, 
by Proposition \ref{prop:3DFace}, the
polynomial $f_w$ corresponding to this face has a singularity. Such a singularity
leads to one fewer complex solution
to the likelihood equations. If $k$ of
these minors vanish then the ML degree
is at most $\deg(X_{2,d}) - k$. In the 
case of $\Delta_{2,4}$ itself, there can 
be at most one singularity and since 
$\deg(X_{2,4}) = 4$, the ML degree can
either be $3$ or $4$.
\end{proof}

Using Theorem \ref{thm: A-det for U2n}, we performed computations to understand the number of vanishing factors of the principal $A$-determinant and their respective ML degrees. 
\begin{table}[t]
\centering
\begin{tabular}{c c c c c c} 
\toprule
$d$ & degree & ML degree & drop & \# principal 4-minors & \# principal 6-minors \\
\midrule
4 & 4 & 3 & 1 & 1\\ 
5 & 11 & 6 & 5 & 5 \\
6 & 26 & 10 & 16 &15 & 1\\
7 & 57 & 15 & 42 & 35 & 7 \\
8 & 120 & 21 & 99 & 70 & 28 \\
\bottomrule
\end{tabular}
\caption{Data of the uniform rank two matroid on up to eight elements. The table shows the degree of the associated toric variety as well as the number of principal $k$-minors of $W$ for $k\ge4$ even. The ML degrees and associated ML degree drops are based on scaled models such that all factors of the principal $A$-determinant vanish. Note that the case $d=8$ has a vanishing principal $8$-minor.} \label{table: uniform matroid data}
\end{table}
Considering $U_d^2$ up to $d = 8$, we have determined a scaling using \texttt{Mathematica} such that all factors of the principal $A$-determinant vanish. The results are shown in Table \ref{table: uniform matroid data}. The corresponding ML degrees were computed using Birch's Theorem \cite[Corollary 7.3.9]{AlgebraicStatistics} and \texttt{HomotopyContinuation.jl} \cite{HomotopyContinuationJL}. 
The ML degree drop given in Table \ref{table: uniform matroid data} equals the number of principal $k$-minors of $W$ for $k \ge 4$ even, including the determinant itself if $d$ is even. Hence the maximum ML degree drop that we were able to achieve is
\begin{equation*}
\sum_{k \ge 4 \textup{ even}} \# k\textup{-minors}.
\end{equation*}
Based on our computations shown in Table \ref{table: uniform matroid data}, we therefore state the following conjecture.

\begin{conjecture}
The minimum ML degree of $X_{2,d}^w$ that can be achieved by choosing a suitable scaling $w$ is  $\binom{d-1}{2}$.
\end{conjecture}

If $d=5$ then we can exhibit all ML degrees between $\textup{deg} (X_{2,5}) = 11$ and the conjectured lower bound of $6$. The corresponding scalings and vanishing principal minors of $W$ are given in Table \ref{table: MLDegreeStratification of U52}. Note that each vanishing principal $4$-minor contributes exactly one singularity according to Proposition \ref{prop:3DFace} and Corollary \ref{cor: singularity}.

\section*{Acknowledgements}
We are grateful to Simon Telen for his substantial help in Theorem \ref{thm: ML degree is a matroid invariant}. We also thank Max Wiesmann for pointing out the proof for Proposition \ref{prop:extension}.
Part of this research was performed while authors were visiting the Institute for Mathematical and Statistical Innovation (IMSI), which is supported by the NSF (Grant No. DMS-1929348).

\vspace{0.58cm}

\begin{table}[H]
\centering
\begin{tabular}{r | c c c c}
\toprule
Scaling & 
$\begin{pmatrix}
0 & 1 & 1 & 1 & 1 \\
  & 0 & 1 & 1 & 1 \\
  &   & 0 & 1 & 1 \\
  &   &   & 0 & 1 \\
  &   &   &   & 0 \\
\end{pmatrix}$
&
$\begin{pmatrix}
0 & \frac{7}{8} & 1 & 1 & 1 \\
  & 0 & 1 & 1 & 1 \\
  &   & 0 & \frac{32}{7} & 1 \\
  &   &   & 0 & 1 \\
  &   &   &   & 0 \\
\end{pmatrix}$
& 
$\begin{pmatrix}
0 & 1 & 1 & 1 & 1 \\
  & 0 & 1 & 1 & 1 \\
  &   & 0 & 4 & 4 \\
  &   &   & 0 & 1 \\
  &   &   &   & 0 \\
\end{pmatrix}$ \\
Discriminants & 
$(*, *, *, *, *, *)$
& 
$(*, *, *, *, *, 0)$
& 
$(*, *, *, *, 0, 0)$\\
ML degree & $11$ & $10$ & $9$ \\
\midrule
Scaling &
$\begin{pmatrix}
0 & 1 & 1 & 1 & 1 \\
  & 0 & 1 & 1 & 1 \\
  &   & 0 & 4 & 4 \\
  &   &   & 0 & 4 \\
  &   &   &   & 0 \\
\end{pmatrix}$
&
\multicolumn{2}{c}{
$\begin{pmatrix}
0 & -3 & \frac{7}{2} & \frac{21}{4}(1-\sqrt{3}i) & \frac{3}{4}(1+\sqrt{3}i) \\
 & 0 & -2 & -3-3\sqrt{3}i & -\frac{3}{7}+\frac{3}{7}\sqrt{3}i \\
 &  & 0 & -7 & -1 \\
 &  &  & 0 & 3 \\
 &  &  &  & 0 \\
\end{pmatrix}$}\\
Discriminants &
$(*, *, *, 0, 0, 0)$
& 
\multicolumn{2}{c}{$(*, 0, 0, 0, 0, 0)$}\\
ML degree & $8$ & 
\multicolumn{2}{c}{$6$} \\
\midrule
Scaling & 
$\begin{pmatrix}
0 & 1 & \frac{1}{4} & 1 & 1 \\
  & 0 & 1 & 1 & 1 \\
  &   & 0 & \frac{1}{4} & \frac{1}{4} \\
  &   &   & 0 & 4 \\
  &   &   &   & 0 \\
\end{pmatrix}$
&
\multicolumn{2}{c}{$\begin{pmatrix}
0 & -4 & 1 & 4 & -\frac{19}{4}-2\sqrt{3} \\
 & 0 & -1 & -4 & \;\;\;\frac{19}{4}-2\sqrt{3} \\
 &  & 0 & 4 & -4 \\
 &  &  & 0 & -3 \\
 &  &  &  & 0 \\
\end{pmatrix}$} \\
Discriminants & 
$(*, *, 0, 0, 0, 0)$& 
\multicolumn{2}{c}{$(0, 0, 0, 0, 0, 0)$}\\
ML degree & $7$ & 
\multicolumn{2}{c}{$6$} \\
\bottomrule
\end{tabular}
\caption{ML degree stratification of the uniform rank two matroid on five elements. The table shows the ML degree and the $A$-disciminants $(\det(W), W^1, W^2, W^3, W^4, W^5)$ where $W^i$ is the principal $4$-minor obtained by deleting the $i$th row and column. We write $*$ to indicate a non-zero value and $0$ to show when the disciminant vanishes.} \label{table: MLDegreeStratification of U52}
\end{table}

\bibliographystyle{alpha}
\bibliography{bibliography}

\vspace{0.5cm}

\noindent{\bf Authors' addresses:}
\smallskip
\small 

\noindent Oliver Clarke,
University of Edinburgh
\hfill {\tt oliver.clarke@ed.ac.uk}

\noindent Serkan Hoşten,
San Francisco State University
\hfill {\tt serkan@sfsu.edu}

\noindent Nataliia Kushnerchuk,
Aalto University
\hfill {\tt nataliia.kushnerchuk@aalto.fi}

\noindent Janike Oldekop,
Technische Universit\"at Berlin
\hfill {\tt oldekop@math.tu-berlin.de}

\end{document}